\documentclass[11pt]{amsart}
\usepackage{amsmath, amsfonts,amsthm,amssymb,amscd, 
verbatim,graphicx,color,multirow,booktabs, caption,tikz,tikz-cd, mathdots,bm}
\usepackage[margin=2cm]{geometry}
\usepackage{tikz-cd}
\usepackage[breaklinks, colorlinks, citecolor=blue]{hyperref}
\hypersetup{colorlinks=true, linkcolor=red,linktocpage=True}

\usepackage{latexsym}
\usepackage{longtable}
\usepackage{epsfig}
\usepackage{amsmath}
\usepackage{hhline}
\usepackage{comment}

\newcommand{\Alt}{\mathop{\mathrm{Alt}}}
\newcommand{\Sym}{\mathop{\mathrm{Sym}}}
\newcommand{\Sn}{\mathrm{Sym}(n)}
\newcommand{\An}{\mathrm{Alt}(n)}

\DeclareMathOperator{\BBF}{\mathbb{F}}
\DeclareMathOperator{\tr}{\mathrm{tr}}

\DeclareMathOperator{\Out}{Out}
\newcommand{\Aut}{\mathop{\mathrm{Aut}}}

\DeclareMathOperator{\out}{Out}
\DeclareMathOperator{\soc}{soc}

\DeclareMathOperator{\rad}{Rad} 

\DeclareMathOperator{\GL}{GL} 
\DeclareMathOperator{\PSU}{PSU}
\DeclareMathOperator{\PGL}{PGL}
\DeclareMathOperator{\PSL}{PSL}
\DeclareMathOperator{\PGU}{PGU}

\DeclareMathOperator{\GU}{GU}
 
\DeclareMathOperator{\PGamL}{P\Gamma L}
\DeclareMathOperator{\GO}{GO}
\DeclareMathOperator{\PGO}{PGO}

\DeclareMathOperator{\POmega}{P\Omega}

\DeclareMathOperator{\PCGamSp}{PC\Gamma Sp}
\DeclareMathOperator{\Sp}{Sp}

\DeclareMathOperator{\PSp}{PSp} 

\DeclareMathOperator{\sym}{Sym}
 
\DeclareMathOperator{\AGL}{AGL}

\DeclareMathOperator{\alt}{Alt}

\renewcommand{\AA}{\mathcal A}

\renewcommand{\SS}{\mathcal S}
\newcommand{\Supp}{\mathrm{Supp}}

\newcommand{\BB}{\mathcal{B}}

\newcommand{\NN}{\mathcal{N}}

\newcommand{\MAGMA}{{\sc Magma} \,}
\newcommand{\MAGMAn}{{\sc Magma}}
\newtheorem{thm}{Theorem}[section]
\newtheorem{mainthm}{Theorem}

\newtheorem{lemma}[thm]{Lemma}
\newtheorem{prop}[thm]{Proposition}

\theoremstyle{definition} 
\newtheorem{notation}[thm]{Notation}
\theoremstyle{definition}
\newtheorem{defn}[thm]{Definition}

\newtheorem{question}[]{Question}

\numberwithin{equation}{section}

\renewcommand{\footnote}{\endnote}
\newcommand{\ignore}[1]{}\makeglossary

\usepackage[nodayofweek]{datetime}
\usepackage{datetime}

\begin{document}
	\title[]{ Base sizes of primitive permutation groups}

\author[M. Moscatiello]{Mariapia Moscatiello}
\address{Mariapia Moscatiello, Dipartimento di Matematica, Università di Bologna, Piazza di Porta San Donato 5, 40126 Bologna, Italy} 
\email{mariapia.moscatiello@unibo.it}

\author[ Colva M. Roney-Dougal]{Colva M. Roney-Dougal}
\address{Colva M. Roney-Dougal, School of Mathematics and Statistics,  The University of St Andrews,\newline    
	North Haugh, St Andrews, Fife, KY16 9SS, Scotland}
\email{colva.roney-dougal@st-andrews.ac.uk}
\date{\today}


\makeatletter
\@namedef{subjclassname@2020}{%
  \textup{2020} Mathematics Subject Classification}
\makeatother


\begin{abstract} Let $G$ be a permutation group, acting on a set
  $\Omega$ of size $n$. A subset $\BB$ of $\Omega$ is a \emph{base} for
  $G$ if 
the  pointwise stabilizer $G_{(\BB)}$ is trivial. Let $b(G)$ be the
minimal size of a base for $G$.  A subgroup $G$ of $\Sn$ is
\emph{large base} if there exist integers $m$ and $r \geq 1$  such
that $\alt(m)^r \unlhd G \leq \sym(m)\wr \sym(r)$, 
where the action of $\sym(m)$ is on $k$-element subsets of
$\{1,\dots,m\}$ 
 and the wreath product acts with product action. In this paper we
 prove that if $G$ is primitive and not large base, then either $G$ is
 the Mathieu group $\mathrm{M}_{24}$ in its natural action on $24$
 points, or 
$b(G)\le \lceil \log n\rceil+1$. Furthermore, we show that
there are infinitely many primitive groups $G$ that are not large base
for which $b(G) > \log n + 1$, so our bound is optimal. 
\end{abstract}

\maketitle

\vspace{-10pt}

\section{Introduction}

Let the permutation group $G$ act on a set
  $\Omega$ of size $n$. A subset $\BB$ of $\Omega$ is a \emph{base} for
  $G$ if 
the  pointwise stabilizer $G_{(\BB)}$ is trivial. Let $b(G, \Omega)$,
or just $b(G)$ when the meaning is clear, be the
minimal size of a base for $G$.  

In the 19th century, a  problem that attracted a lot of attention was
that of bounding the order of a finite primitive permutation group. It is easy to show that
$|G|\le n^{b(G)}$, so one can 
find an upper bound on the order of a permutation group
by bounding the minimal base size. 
One of the earliest results in this direction is a theorem of
Bochert \cite{Bo} from 1889, which states that 
if $G$ is a primitive permutation group of degree $n$ not containing
the alternating group $\alt(n)$, then $b(G)\le n/2$.

Bases also arise naturally in other contexts, which also benefit from
good upper bounds on base size. For example,
they have been used extensively in the computational study of finite permutation groups, where the problem of calculating base sizes has
important practical applications. The knowledge of how an element
$g$ of $G$ acts on a base $\BB$ completely determine the action of $g$ on
$\Omega$, so once a base and a
related data-structure called a strong generating set are known for
$G$, we may store elements of $G$ as $|\BB|$-tuples, rather than
as permutations, of $\Omega$.

A permutation group $G$ is 
\emph{large base} if there exist integers $m$ and $r \geq 1$  such
that $\alt(m)^r \unlhd G \leq \sym(m)\wr \sym(r)$, 
where the action of $\sym(m)$ is on $k$-element subsets of
$\{1,\dots,m\}$ for some $k$, 
 and if $r > 1$ then  $G$ has product action.  Note that this
 includes the natural actions of $\An$ and $\Sn$. 

Using the Classification of Finite Simple Groups (CFSG),
and building on earlier work by Cameron \cite{cam}, in 1984 Liebeck \cite{Lie1} 
proved  the
remarkable result that if
$G$ is a primitive group of degree $n$ that
 is not large base, then $b(G)\le 9 \log n$.
(In this paper, all logarithms are to base $2$, unless otherwise indicated.) 
Much more recently, Liebeck, Halasi and Mar\'oti showed in \cite{HaLiMa}
that for most non-large-base primitive groups $G$, the base size $b(G) \leq 2 \lfloor \log n
\rfloor + 26$; the second author and Siccha then noted in
\cite{CMRDSiccha} that
this bound applies to all primitive groups that are not large base. 

The main result of this paper is as follows. 

\begin{mainthm}\label{thm:main}
	Let $G$ be a primitive permutation group of degree $n$. If $G$
        is not large base, then either $G$ is the Mathieu group
        $\mathrm{M}_{24}$ in its $5$-transitive action of degree $24$,
        or $b(G) \leq \lceil \log   n\rceil+1$. Furthermore, there are infinitely many such
        groups $G$ for which $b(G) > \log n + 1$. 
\end{mainthm}

If $G$ is $\mathrm{M}_{24}$ in its
$5$-transitive action of degree 24  then $b(G)=7$.
In Theorem~\ref{thm:final} we shall completely classify the
non-large-base primitive groups $G$ for which $b(G) > \log n + 1$:
there is one infinite family, and three Mathieu groups. 


Our notation for groups is generally standard: for the orthogonal
groups, by $\GO_d^\varepsilon(q)$ we denote the full isometry group of
our standard quadratic form of type $\varepsilon$, as given in
Definition~\ref{def:standardform}.

\begin{defn} \label{subspaceaction}
Let $G$ be almost simple with socle $G_0$, 
 a classical group with natural module $V$
 over a field  of characteristic $p$. 
A subgroup $H$ of $G$ not containing $G_0$ is a \emph{subspace}
subgroup if for each maximal subgroup $M$ of $G_0$ containing $H\cap
G_0$ one of the following holds. 
\begin{itemize}
	\item[(1)] $M=G_U$ for some proper nonzero subspace $U$ of
          $V$, where $U$ is either totally singular, or non-degenerate, or, if
          $G$ is orthogonal and $p = 2$, a nonsingular 1-space. If
          $G_0 = \PSL_d(q)$ then we shall consider all subspaces of
          $V$ to be totally singular.
	\item[(2)] $G_0 =\Sp_{d}(2^f)$ and $M\cap G_0 =\GO^{\pm}_{d}(2^f)$.
\end{itemize}

A transitive action of $G$ is a \emph{subspace
  action}  if the point stabiliser is a subspace subgroup of $G$.
\end{defn} 

\begin{defn}\label{def:standard}
	Let $G$ be almost simple with socle $G_0$. 
A transitive action of $G$ on $\Omega$ is \emph{standard} 
 if, up to equivalence of actions, one of the following holds, and is
 \emph{non-standard} otherwise.
\begin{enumerate}
\item $G_0=\alt(\ell)$ and $\Omega$ is an orbit of subsets or partitions of $\{1,\dots,\ell\}$;
\item $G$ is a classical group in a subspace action.
\end{enumerate}
\end{defn}

Cameron and Kantor conjectured  in \cite{cam, camka} that there exists an absolute
constant $c$ such that if $G$ is almost simple with a faithful primitive non-standard action on a finite set $\Omega$
then
$b(G) \le  c$. 
In \cite[Theorem 1.3]{LIETYPE}, Liebeck and Shalev  proved this
conjecture, but without specifying the 
constant $c$. Later, in a series of papers \cite{
Burness07, BurnessOBrienWilson10, BGS11},
 Burness
and others proved  that $b(G)\leq 7$, with equality if and only if $G$ is
$\mathrm{M}_{24}$ in its $5$-transitive action of
degree 24; that is, the Cameron-Kantor conjecture is true with
the constant $c = 7$.

In stark contrast with the non-standard case, the base size
of a group with a standard action can be arbitrarily large. The bulk
of this article  therefore concerns such actions. For many of the
standard actions we shall use results due to Halasi,
Liebeck and Mar\'oti \cite{HaLiMa}, however
we sometimes  require more precise bounds.

\begin{notation}\label{not:SandN}
  Let $G$ be a classical group, with natural module $V$. We shall
  write $\SS(G, k)$ for a $G$-orbit of totally singular subspaces of
  $V$ of 
  dimension $k$, and $\NN(G, k)$ for a $G$-orbit of
  non-degenerate or non-singular subspaces of $V$ of dimension $k$. For the
  orthogonal groups, let $W$ be a space in the orbit if $dk$ is even, and the
  orthogonal complement of such a space if $dk$ is odd.  Then
  we write $\NN^\epsilon(G, k)$,
  with $\epsilon \in \{+, -\}$, to indicate that the restriction
  of the form to $W$ is of type $\epsilon$:  if $d$ is odd then the symbol
  $\NN(G, k)$ is not used, since $k$ or $d-k$ is even. 
\end{notation}    

The next result is a key tool in the proof of Theorem~\ref{thm:main},
but may be of independent interest. It will follow immediately from the results
in Section~\ref{sec:onetwo}: bounds for smaller dimensions
may be found there. 

\begin{mainthm}\label{thm:base_summary}
  Let $G$ be one of $\PGL_d(q)$, $\PGU_d(q)$, $\PSp_d(q)$, or
  $\PGO^\varepsilon_d(q)$. Let $k \in \{1, 2\}$, and let $\Omega$ be $\SS(G, k)$ or
  $\NN^{\epsilon}(G, k)$, with $\epsilon$
  either $+$, $-$, or blank. 
  \begin{enumerate}
     \item Assume that $d \ge 5$, that $G$ is $\PGL_d(q)$, and that $k = 2$. 
      Then $b(G) \leq
      \lceil d/2 \rceil + 2$. 
    \item Assume that $d \ge 3$, that $G$ is $\PGU_d(q)$ or
      $\PSp_d(q)$, 
      and that
      $k = 1$. Then $b(G) \leq d$.
    \item Assume that $d \ge 6$, that $G$ is $\PGO^\varepsilon_d(q)$, and that $k =
      1$. Then $b(G) =
      d-1$. 
    \item Assume that $d \ge 7$, that $G$ is not $\PGL_d(q)$, and that $k = 2$. Then
      $b(G) \le \lceil d/2 \rceil$. 
      \end{enumerate}
    Additionally, if $q$ is even, $d \ge 6$, and $\Omega$ is the right coset space
      of $\GO^{\pm}_d(q)$ in $G = \Sp_d(q)$, then $b(G) = d$.
 \end{mainthm}

 We shall prove this result by giving explicit bases of the stated
 size. These bounds are very similar to those proved by Burness,
 Guralnick and Saxl in \cite{BGS17} for algebraic groups, although we
 consider the full projective isometry group.
 Unfortunately we were not able to
 directly transfer many of their proofs over to the finite case, but we
 have taken some inspiration from their arguments. 

We notice in passing that the value of $b(G, \Omega)$ for $\Omega$ the
right coset space of $\GO^{\pm}_d(q)$ in $G = \Sp_d(q)$ is only one
less than the value of the largest irredundant base size for this
action, as proved in \cite{GillLodaSpiga}: in general these two quantities
can be very different.

\bigskip

\noindent \textbf{Acknowledgements} 
The authors would like to thank the
  Isaac Newton Institute for Mathematical Sciences for support and
  hospitality during the programme ``Groups, Representations and
  Applications: New perspectives'', when  work on this paper was undertaken. This work was supported by:
EPSRC grant numbers EP/R014604/1 and EP/M022641/1.
We are grateful to Professor Liebeck
for several helpful suggestions.

\section{Explicit bases for some subspace actions}\label{sec:onetwo}

Let $G$ be a finite almost simple classical group with natural module $V$.
In this section we present explicit bases for the action of $G$ on a $G$-orbit of totally
singular, non-degenerate, or non-singular  one- or two-dimensional subspaces of
$V$,  and for the action of $\Sp_d(q)$ on the right cosets of
$\GO^{\pm}_d(q)$, with $q$ even.

\begin{defn}\label{def:standardform}
Let  $\BBF = \BBF_{q^2}$ in the unitary case, and $\BBF = \BBF_q$
otherwise, and let $\sigma$ be the automorphism of
$\mathbb{F}$ mapping $x \mapsto x^q$. Write $\BBF^\ast$ for the
non-zero elements of $\BBF$. 

We fix our standard classical forms and bases on $V = \BBF^d$. Our
standard basis for $\GL_d(q)$ will be $(v_1,\dots, v_d).$ If $d = 2a$ then our standard unitary and symplectic forms $B$ have
basis
$(e_1,\dots,e_{a}, f_{1},\dots, f_a)$
 whilst if $d = 2a +1$ then our standard unitary form has basis
 $(e_1,\dots,e_{a},f_{1},\dots, f_{a}, x)$. In both cases, for
 all $i$ and $j$ we set
$B(e_i,e_j)=B(f_i,f_j)=0$, $B(e_i,f_j)=\delta_{i,j}$ (the Kronecker $\delta$),
$B(e_i,x)=B(f_i,x)=0$, and $B(x,x) = 1.$
	
Our standard quadratic form $Q$, with symmetric bilinear form $B$,
has basis
	\begin{align*}
	\begin{cases} (e_1,\dots,e_{a},f_{1},\dots, f_{a}) &
          \mbox{if $d = 2a$ and $Q$ is of $+$ type,}\\ 
	(e_1,\dots,e_{a},f_{1},\dots, f_a, x, y)  & \mbox{if $d =
          2a+2$ and $Q$ is of $-$ type,}\\ 
	(e_1,\dots,e_{a},f_{1},\dots, f_{a}, x) & \mbox{if $d = 2a+1$,}\\
	\end{cases}
	\end{align*} 
	where for all $i$ and $j$ we set 
$Q(e_i)=Q(f_i)=0, B(e_i,f_j)=\delta_{i,j}$, 
$B(e_i,x)=B(f_i,x)=B(e_i,y)=B(f_i,y)=0$,
	$Q(x) = B(x,y)=1$ and $Q(y)=\zeta$, where $X^2+X+\zeta\in
        \mathbb{F}[X]$ is irreducible. We will work, at
        times, with orthogonal groups of odd dimension in
        characteristic two, and this is our standard form in
        this case as well: see, for example, \cite[p139]{Taylor} for more information.

A pair $(u, v)$ of vectors in $V$ is a \emph{hyperbolic pair} if $B(u,
u) =  B(v, v) = 0$, 
$B(u, v) =1$, and (in the orthogonal case) $Q(u) = Q(v) = 0$.
\end{defn}

We now collect a pair of elementary lemmas.
The first two parts of the following are well known, and the
third is easy. By the \emph{support} of a vector $v$, denoted $\mathrm{Supp}(v)$,
we mean the set of basis vectors for
which the coefficient is nonzero.

\begin{lemma}\label{lem:linone}
  Let $W= \BBF_q^d$ with basis $w_1,\dots, w_d$,
  let $H=\GL_d(q)$, and let $\AA=\{\langle w_1\rangle
,\dots, \langle w_d\rangle\}$. 
\begin{itemize}
\item[(1)] $H_{(\AA)}$ is a group of diagonal matrices, and
  is trivial when $q = 2$.  
\item[(2)] For all $\mu:= (\mu_1, \ldots,
  \mu_d) \in (\BBF_q^\ast)^d$, let $\AA(\mu)=\AA\cup \{\langle
\mu _1 w_{1}+\cdots+\mu_dw_{d}\rangle \}$. Then $H_{(\AA(\mu))} =
Z(\GL_d(q))$.
\item[(3)] Let $T = \langle u, v \rangle \leq W$, and let 
$g \in H$ be such that $T^g = T$. 
If there exists an $s\in \Supp(v)$ such that for all $t\in \Supp(u),$ the vector $s\notin  \Supp(tg),$
then  $\langle u \rangle^g = \langle u  \rangle$.
\end{itemize} 
\end{lemma}	

In the presence of a non-degenerate form, we can make stronger statements. 

\begin{lemma} \label{lem:orthogonal}
Let  $B$ be a non-degenerate sesquilinear 
form on $V = \BBF^d$, with $d > 2$. Let $u,v\in V$ be such that $\langle u, v
\rangle$ is non-degenerate, and let $g$ be an isometry of $V$ such that
$ug=\alpha u$ for some $\alpha 
\in \BBF^\ast$. 
\begin{enumerate}
\item Assume that $vg= \beta v$, for some $\beta \in \BBF^\ast$. If
  $(u, v, w)$ are such that $0 \neq w\in \langle u, v
\rangle^{\perp}$,  and $g$ stabilises $\langle
\gamma_1u+\gamma_2v+\gamma_3w\rangle$ for some
$\gamma_i \in \BBF$ with $\gamma_1 \gamma_3 \neq 0$, then
$wg = \alpha w$. Furthermore, if $\gamma_2 \neq 0$ then
$\beta = \alpha$, and if, in addition, $B(u, v) \neq 0$ then $\alpha =
\alpha^{-q}$. 

\item Assume instead that $B$ is symmetric, and that
$(u, v)$ are a hyperbolic pair. 
        If $v g \in \langle u,v\rangle$, then
        $vg=\alpha^{-1}v.$ 
\end{enumerate}
\end{lemma}

\begin{proof}
(1). Since $\langle u, v \rangle$ is non-degenerate, $g$ preserves the
decomposition $V =
\langle u, v \rangle \oplus \langle u, v \rangle^\perp$. 
Let $\{w = w_3, w_4, \dots,w_d\}$ be a basis of $\langle u, v \rangle^\perp$. Then
there exist $\lambda_3, \ldots, \lambda_d$
such that  
$wg=\sum_{i = 3}^d \lambda_i w_i$. 
Further, there exists $ \mu \in \mathbb{F}_q$ such that  \[\mu
(\gamma_1u+\gamma_2v+\gamma_3w)=(\gamma_1u+\gamma_2v+\gamma_3w)g
=\gamma_1 \alpha u+\gamma_2 \beta v+\gamma_3
(\sum_{i = 3}^d \lambda_i
w_i).\]
Hence $\mu = \alpha = \lambda_3$ and $\lambda_i=0$ for $4 \leq i \leq
d$. Furthermore, if $\gamma_2 \neq 0$ then $\beta = \alpha$. The final
claim is clear.

\noindent
(2).  Let $vg=\beta u+\gamma
        v$.  From $1= B(u, v) = B(ug, vg)= \alpha\gamma$, we deduce that 
        $\gamma = \alpha^{-1} \neq
        0$.  Then $$0=Q(v)=Q(vg)=Q(\beta u+\gamma v)= \beta \gamma$$
        implies that $\beta=0$.  
\end{proof}

\subsection{Totally singular subspaces}\label{subsec:sing12}

In this subsection we consider the unitary, symplectic and orthogonal
groups acting on $\mathcal{S}(G,k)$ for  $k\in \{1,2\}$, where
$\SS(G, k)$ is as in Notation~\ref{not:SandN}.
We shall use without further comment the fact that the trace map from
$\BBF_{q^2}$ to $\BBF_q$, given by 
$\tr(\alpha) = \alpha + \alpha^q$, is surjective.

\begin{lemma} \label{lem:sing_one}
Let $G$ be $\PGU_d(q)$, $\PSp_d(q)$, $\PGO^\varepsilon_d(q)$, with $d
\ge 5$ if $G$ is orthogonal, and $d \ge 3$ otherwise, and let
$\Omega=\mathcal{S}(G, 1)$.
Then the set $\BB$ in
Table~\ref{tab:sing_one} is a base for the action of $G$ on
$\Omega$. In particular, $b(G) \le d$ and if $G$ is orthogonal then $b(G) \le d-1$. 
\end{lemma}

\begin{table}\caption{Bases for $\SS(G, 1)$} \label{tab:sing_one}
  Let $V_i=\langle e_1+e_i\rangle,$ $W_i=\langle e_1+f_i\rangle,$
    and $T=\langle -e_1+f_1+x\rangle$
  \begin{tabular}{| l | l | l |}
    \hline
    $G$ & $\mathcal{B}$ & Comments\\
    \hline \hline
                    $\PGU_{2a+1}(q)$ &
     $\{\langle e_1\rangle,\langle f_1\rangle, V_i, W_i, \langle
                  e_1+\mu f_1 + x\rangle \mid 2\le i\le a\}$ & $\tr(\mu) = -1$\\
    $\PGU_{2a}(q)$, $\PSp_{2a}(q)$ &
     $\{\langle e_1\rangle,\langle f_1\rangle, V_i, W_i\mid 2\le i\le a\}$ &\\
   $\PGO^+_{2a}(q)$& 
   $\{\langle e_1\rangle,\langle f_1\rangle, V_i, W_j, \mid 2\le i \le a, 2\le j \le a-1\}$ & \\ 
    $\PGO_{2a+1}(q)$ & \{$\langle e_1\rangle, \langle f_1\rangle, V_i, W_j,T \mid 2\le i \le a, 2\le j \le a-1\}$ & \\  
    $\PGO^-_{2a+2}(q)$ & $\left\{ \langle e_1\rangle,\langle f_1\rangle, V_i, W_j,T, \langle -\zeta e_1+f_1+y\rangle  \mid 2\le i \le a, 2\le j \le a-1\right\}$ & $\zeta$ from Defn~\ref{def:standardform}
 \\  
    \hline
  \end{tabular}
  \end{table}

\begin{proof}
Let $H = \GU_d(q)$, $\Sp_d(q)$, or $\GO^\varepsilon_d(q)$.
 First let $\BB$ be one of the sets listed in
 Table~\ref{tab:sing_one}. 
A straightforward calculation shows that each subspace in $\BB$ is 
 singular, so $\BB \subseteq \Omega$. Let $g \in H_{(\BB)}$.
We shall show that $g$ is scalar, from which the result will follow. 
To do so, we shall repeatedly apply Lemma~\ref{lem:orthogonal}(1), with $(u,v,w)$  set to be equal to various triples of vectors.

For $\PGU_3(q)$ it suffices to apply
Lemma~\ref{lem:orthogonal}(1) to $(e_1,f_1,x)$. So we can assume that $d\ge 4$. Apply Lemma \ref{lem:orthogonal}(1), first to 
$(e_1,f_1,e_i)$ and then to $(e_1,f_1,f_j)$
to see that there exists $\alpha\in \BBF$ such that

\begin{equation}\label{eq:bo}
  e_ig=\alpha e_i, \,\; \;f_jg=\alpha f_j,\;\quad \mbox{for $1\le i\le
    a$ and}
  \left\{\begin{array}{ll}
            2\le j\le a & \mbox{if $H$ is $\GU_{d}(q)$ or $\Sp_{d}(q),$}\\
            2\le j\le a-1 & \mbox{if $H$ is orthogonal.}
            \end{array} \right. 
\end{equation}
Now, $B(e_{1}g,f_{1}g)=1$ yields 
\begin{align}\label{eq:bo2}
f_1g=\alpha^{-q} f_1.
\end{align}
For $\PGU_{2a+1}(q)$ the result follows by applying Lemma~\ref{lem:orthogonal}(1) to $(e_1,f_1,x)$. 
For $\PGU_{2a}(q)$, $\PSp_{2a}(q)$ and $\PGO^+_{2a}(q)$,   
we deduce from $B(e_2g, f_2g) = 1$ that
$\alpha=\alpha^{-q},$ hence if $G$ is not orthogonal then $g$ is scalar.

For $\PGO_{2a+1}(q)$, applying Lemma~\ref{lem:orthogonal}(1) to
$(e_1,f_1,x)$ shows that $xg=\alpha x = \pm x$.
Similarly, for $\PGO^-_{2a+2}(q)$, applying
Lemma~\ref{lem:orthogonal}(1) to both $(e_1,f_1,x)$ and $(e_1,f_1,y)$
yields $xg=\alpha x = \pm x$ and $yg=\alpha y.$
Combining these with (\ref{eq:bo}) and (\ref{eq:bo2}), we deduce that
if $H$ is orthogonal then $g$ stabilizes $\langle
e_{a},f_{a}\rangle^\perp$,  and so  
stabilizes $\langle e_a,f_a\rangle.$ Then
Lemma~\ref{lem:orthogonal}(2) shows that $f_ag=\alpha f_a$, so
$g$ is scalar.  
\end{proof}

\begin{lemma} \label{lem:twospacelin}
Let  $G = \PGL_d(q)$ and let $\Omega= \mathcal{S}(G,2).$
Then the set $\BB$ in Table~\ref{tab:k=2sing} is a base for the action of $G$ on $\Omega$. In particular
$b(G)\le \lceil\frac{d}{2}\rceil+2$ when $d \ge 5$, and  $b(G)\le 5$ when $d = 4$.
\end{lemma}

\begin{proof}
                Let $g  \in \GL_d(q)_{(\BB)}$: we shall show that $g$
                is scalar. 
The arguments for $d = 4$ are similar to, but easier than, those that
follow, so let $d\ge 5$, and
 let  
$X  =X_1\oplus\dots\oplus X_{a-1}$. 
Then $g$ stabilises
$Y_1\cap X  =\langle  v_2+v_4+\dots+v_{2a-2}
\rangle$. 
 Hence there exists $\alpha \in \BBF_q$ such that
\begin{equation*}\label{eq:evenbutthelast}
v_{2j}g=\alpha v_{2j}, \, \, \text{for}\; 1 \le j\le a-1.
\end{equation*}
Furthermore, $g$ stabilises
	$X_1\cap Y_2=\langle v_1\rangle$, and hence $v_1g=\beta v_1$
        for some $\beta \in \BBF.$  Now, this and the fact that $g$
        stabilizes $X_2 \oplus \dots \oplus X_{a}
        =\langle v_3, v_4,\dots, v_{2a-1},v_{2a}\rangle$ 
        means that
        we may apply Lemma~\ref{lem:linone}(3), with $u =
        v_1+v_3+\dots+v_{2a-1}$, $v=v_2+v_4+\dots+v_{2a-2},$ and $s=v_2$ to deduce that
\begin{align*}
	v_{2i-1}g=\beta v_{2i-1},  \;\text{for}\, 1\le i\le a.
		\end{align*}
Now, 
$v_d g \in \langle
v_{d-1}, v_d \rangle$ if $d=2a$ is even (and $v_d g= \beta v_d$ otherwise), so
Lemma~\ref{lem:linone}(3), with $T =
Y_2$, $u = v_3 + v_{2a-2} + v_d$  and $s = v_1$ yields 
$\langle
u\rangle^g=\langle
u\rangle$, so $\alpha=\beta$ and
$g$ is scalar.
\end{proof}

\begin{table}\caption{Bases for 
    $\mathcal{S}(G, 2)$}\label{tab:k=2sing}
  Let $a = \lceil d/2 \rceil$, 
      $X_i = \langle v_{2i-1}, v_{2i}\rangle$ with $v_{d+1} = v_1$, $Y_1= \langle v_1 + v_3 + \cdots + v_{2a-1}, v_2
              + v_4 + \cdots + v_{2a-2} \rangle$,             
$V_1 = \langle e_1, e_2 \rangle, V_2 =
   \langle f_1, f_2\rangle $, $W_i=\langle e_1+e_i,
    e_2-f_1+f_i\rangle$, and $\AA = \{V_1, V_2,  W_i \ : \ 3 \leq i \leq
    a  - 1\} $.    
\begin{tabular}{| l| l| l|}
\hline
$G$ & $\mathcal{B}$ & Notes \\
  \hline \hline
  $\PGL_4(q)$ & $\{X_1, X_2, Y_1, \langle v_2, v_4 \rangle, \langle v_1 +
                v_2, v_3 \rangle\}$ & \\
  \hline
$\PGL_d(q)$ & $\{X_i, Y_1, Y_2 = \langle v_1, v_3
              + v_{2a-2} + v_d \rangle \ : \ 1 \le i \le a \}$ & $d \ge 5$ \\
  \hline
$\PGU_4(q)$ & $\{V_1, V_2, \langle e_1+\mu f_1, e_ 2+\mu f_2\rangle, \langle e_1,f_2\rangle, \langle e_1-e_2,f_1+f_2\rangle\}$ &$\tr(\mu)=0$ \\
\hline
$\PSp_4(q)$ & $\{V_1, V_2, \langle e_1+f_1+f_2, e_2+f_1\rangle, \langle e_1+f_2,
e_2+f_1+f_2\rangle\}$ &$q$ even \\
&$\{V_1, V_2, \langle e_1+f_1+f_2, e_2+f_1\rangle, \langle e_1+f_2, e_ 2+ f_1\rangle \}$& $q$ odd\\
\hline
$\PGU_5(q)$ & $\{V_1, V_2, \langle -e_2+\lambda f_2+x,f_1\rangle,
               \langle -e_1+\lambda f_1+x, f_2\rangle\}$ &   $\tr(\lambda) =1$ \\
\hline
$\PSp_6(q), \; \PGU_6(q)$ & $\{V_1, V_2, \langle e_1 + e_3, e_2 - f_1 + f_3 \rangle, \langle e_1-e_2,
                            f_1+f_2\rangle\}$ &\\
  \hline $\PGU_{2a}(q)$, $\PSp_{2a}(q)$, $\PGO^+_{2a}(q)$
    & $\AA \cup \{V_3= \langle e_1 + e_a, e_2
                        - e_a -  f_1 + f_2 + f_a\rangle \}$ & $a \ge 4$\\
  \hline
  $\PGU_{2a-1}(q)$ & $\AA \cup \{V_4 = \langle
  -e_1 + \lambda f_1 + x, e_3 +f_2 \rangle\}$ & $\tr(\lambda) =
                                                  1$, $a \ge 4$\\
  \hline
  $\PGO_{2a-1}(q)$ & $\AA \cup \{V_5 = \langle -e_1 + f_1 + x, e_3 +  f_2 
                     \rangle\}$ &  $a \ge 4$\\
  \hline
  $\PGO^-_{2a}(q)$& $\AA \cup \{V_6 = \langle -e_1  + e_2 + f_1 + x,
                    -\zeta e_1 + f_1 + \zeta f_2 + y \rangle\}$ & $a \ge
  4$ \\
  \hline
\end{tabular}
\end{table}

\begin{lemma} \label{lem:twosing}
	Let $G \in \{\PGU_d(q), \PSp_d(q),
        \PGO_d^{\varepsilon}(q)\}$ with $d\ge 4,$ and $d\ge 7$ if $G$
        is orthogonal, let  $\Omega = \SS(G, 2)$ and let $b = b(G)$.  Then the set $\BB$ in
Table~\ref{tab:k=2sing} is a base for the action of G on $\Omega$. In particular, if $d \ge 7$ then $b \le \lceil\frac{d}{2}\rceil$, if $G = \PGU_4(q)$
then $b \le 5$, 
whilst otherwise, if $d\le 6$ then $b\le 4.$
\end{lemma}

\begin{proof}
 The arguments for $d \le 6$ are similar to, but more straightforward
than, those that follow, so we shall assume
that $d \ge 7$, so that $a = \lceil d/2 \rceil\geq 4$.

Let $H$ be $\GU_d(q)$, $\Sp_d(q)$ or
$\GO^\varepsilon_d(q)$, and let $g\in H_{(\BB)}.$ It is straightforward to verify that 
$\BB \subseteq \Omega$.
Since $V_i^g = V_i$ for $i \in \{1, 2\}$, there exist 
$\alpha_i,\beta_i, \gamma_i, \delta_i \in \BBF$ such that 
\begin{equation}\label{eq:second0}
\begin{array}{cc}
e_1g=\alpha_1e_1+\alpha_2 e_2,& e_2g=\beta_1 e_1+\beta_2 e_{2}\\
f_1g=\gamma_1 f_1+\gamma_2 f_2,& f_2g=\delta_1 f_1+\delta _2f_{2}.
\end{array}
\end{equation}
Let $\AA$ be as in
Table~\ref{tab:k=2sing}, and let $X = \langle \AA \rangle$.
We shall first show that 
\begin{equation}\label{eq:gen2}
  e_i g = \alpha_1 e_i\;\,\mbox{and}\;\, f_ig = \beta_2 f_i \,\;  \mbox{ for } i = \{1, 3, 4,
  \ldots, a-1\}, \quad e_2 g
  = \beta_2 e_2, \quad f_2 g = \alpha_1 f_2.
\end{equation} 

Let $U = V_1 \oplus V_2$, 
and let $W= U^\perp$, so that $W^g =W$. 
For $3 \le i \le a-1$, the element $g$ stabilises $U_i:= \langle V_1, V_2,
W_i \rangle$, and so stabilises $U_i \cap W = \langle e_i, f_i
\rangle$.
Then Lemma~\ref{lem:linone}(3), with 
$u=e_1+e_i$, $v=e_2-f_1+f_i,$ and $s=f_1$ 
 shows 
that there exists $\eta\in \BBF$ such that
$(e_1+e_i)g =\eta(e_1+e_i) =\alpha_1e_1+\alpha_2 e_2+  e_ig$, where the last equality holds by \eqref{eq:second0}. 
Hence
 \eqref{eq:gen2} holds for $e_i$ for $i \neq 2$.
Similarly, for $3 \le i \le a-1,$ there exist $\eta,\rho \in \BBF$ such that
\[
\begin{array}{rl}
(e_2-f_1+f_i)g&=\eta(e_1+e_i)+ \rho (e_2-f_1+f_i)\\
&=\beta_1e_1+\beta_2 e_2-\gamma_1 f_1-\gamma_2 f_2+ f_ig.
\end{array}
\]
Equating coefficients, we deduce  from $f_ig \in \langle e_i, f_i
\rangle$ that $\gamma_2=0$ and  $\beta_2 = \gamma_1$, so that $f_1 g=
\beta_2 f_1$,  and also deduce that $f_i g = \beta_1
e_i + \beta_2 f_i$ for $3 \leq i \leq a-1$. 
For $i \in \{1, 2\}$, let $A_i = \langle e_i, f_i \rangle$. 
Then $A_1^g= A_1$,
so $g$ stabilises $A_1^\perp \cap U = A_2$, and consequently stabilizes
$V_1\cap A_2=\langle e_2\rangle$ and
$V_2\cap A_2= \langle f_2\rangle$, and so
$\beta_1= \delta_1 = 0$. 
Finally, $B(e_1g,f_1g)=B(e_2g,f_2g)=1$  yields 
\[
\alpha_1 =\beta_2^{-q} ,\; \mbox{and}\;\,\beta_2 =\delta_2^{-q},	
\]
hence $\alpha_1 =\delta_2$, and so \eqref{eq:gen2} follows.

\medskip

We now complete the proof that $g=\alpha_1 I_d,$ so
$\BB$ is a base for $G$. If $d = 2a-1$ then (\ref{eq:gen2}) yields
 $(X^\perp)^g = \langle x \rangle^g = \langle x
\rangle$. Let $u=-e_1+f_1+x$ if $G$ is orthogonal and $u=-e_1+\lambda f_1+x$ otherwise. Then 
Lemma~\ref{lem:linone}(3), with $v=e_3 + f_2$ and $s=f_2$, shows
that $\langle u \rangle^g = \langle u \rangle$,
and so $g = \alpha_1 I_{d}$, as required.

If $H=\GO^-_{2a}(q)$
then
$(X^\perp)^g = \langle x, y \rangle^g = \langle x, y \rangle$. We deduce  from \eqref{eq:gen2} and
Lemma~\ref{lem:linone}(3), with $T = V_6$, $u=-e_1+e_2+f_1+x$ and $s=f_2,$
that $\langle u \rangle^g = \langle u \rangle$, 
and so  $\alpha_1 =\beta_2$ and $xg =\alpha_1 x$.
Now considering $u=-\zeta e_1+f_1+\zeta f_2+y$ and $s = e_2$
shows that
$g$ is scalar. 

\bigskip

\noindent
Finally, consider $\PGU_{2a}(q),
\PSp_{2a}(q)$ and $\PGO^+_{2a}(q)$.
From \eqref{eq:gen2} we see that $\langle e_a, f_a \rangle^g = \langle
e_a, f_a \rangle$. Then, by Lemma~\ref{lem:linone}(3), with
$T = V_3$, $u=e_1+e_a$, and $s = e_2$, we deduce that 
$e_ag = \alpha_1 e_a.$ 
Moreover,
letting $u=e_2-e_a - f_1+f_2+f_a,$ $v=e_1+e_a$, and $s=e_1,$
we see that  $\langle u \rangle^g = \langle  u \rangle$, and so
$g=\alpha_1I_{2a},$ as required.
\end{proof}

\subsection{Non-degenerate subspaces} \label{subsec:nondeg12}

In this subsection we consider $\NN^{\epsilon}(G,k)$, where $k \le
2$ and $\NN^\epsilon(G, k)$ is as in
Notation~\ref{not:SandN}.


\begin{lemma}\label{lem:nondeguni1}
	Let $d\ge 3$, let $G = \PGU_d(q)$, and  let $\Omega = \NN(G,
        1)$. 
       Then the set $\BB$ in
Table~\ref{tab:nondeg1} is a base for the action of $G$ on $\Omega$, so $b(G)\le d.$
\end{lemma}

\begin{proof}
First assume that either $d$ is odd or $q > 2$.
Let $\alpha$ be a primitive element of $\BBF^\ast$. Then for at least
one value of $\mu$ in $\{\alpha, \alpha^{-1}, \alpha^2\}$ the vector
$v(\mu) =  v_1 +\dots+  v_{d-1} +\mu v_d$ is non-degenerate, so $\BB
\subseteq \Omega$. 
 Let $g\in \GU_{d}(q)_{(\BB)}$ and
 $U=\langle v_1 , \dots, v_{d-1}\rangle$.
 Since $U$ is non-degenerate, $(U^{\perp})^g=\langle v_d\rangle^g=\langle v_d\rangle,$
and hence $g$ is
diagonal by Lemma~\ref{lem:linone}(1). Then $g$ also stabilises
$\langle v(\mu) \rangle$, and so is scalar, by Lemma~\ref{lem:linone}(2).
	
	For  $q=2$ and $d$ even, $g$ stabilises
	$\langle v_1 , v_2\rangle^\perp = \langle v_3,\dots,
        v_d\rangle.$ Therefore
        Lemma~\ref{lem:orthogonal}(1), applied to  $(v_1,v_2,v_i)$, for $3\leq
        i\le d$, 
shows that $\GU_d(q)_{(\BB)}$ is scalar.
\end{proof}

When $q$ is odd,  $\PGO^\varepsilon_d(q)$ has
two orbits of non-degenerate $1$-spaces. If $d$ is even then the orbits can be distinguished by considering the
discriminant of the restriction of the quadratic form to
the subspace, and the actions on the two orbits are equivalent, so it is
enough to consider one of them.  If $d$ is odd then the orbits can be
distinguished by the sign of the restriction of the form to the
orthogonal complement.

\begin{lemma}\label{lem:nondeg1orth}
  Let $d\ge 4$, let $G= \PGO_d^\varepsilon(q)$ with $\varepsilon=-$ if
  $d=4$,
  and let $\Omega$ be a
  $G$-orbit of non-degenerate or non-singular 
  $1$-spaces.
 Then, up to equivalence,  the set $\BB$ in
Table~\ref{tab:nondeg1} is a base for the action of $G$ on $\Omega$. 
In particular, if $d\ge 6$ then $b(G)\leq d-1$,  $b(\PGO^-_4(q)) \le 3$ if $q \neq 3$, 
  and  $b(\PGO_5(q))\le 5.$  In addition, $b(\PGO^-_4(3)) = 4$.
\end{lemma}

\begin{table}\caption{Bases for $\NN(\PGU_d(q), 1)$ and 
    $\mathcal{N}^\epsilon(\PGO^\varepsilon_d(q),
    1)$}  \label{tab:nondeg1}

  For $\PGU_d(q)$, let
 $(v_1,\dots,v_d)$ be an orthonormal basis of $V$. 
 \begin{tabular}{| l| l| l|}
    \hline
 $d$ and $q$ & $\mathcal{B}$ & Comments \\
    \hline \hline
$d$ odd or $q > 2$&  $  \{ \langle v_1\rangle,\dots, \langle v_{d
                    -1}\rangle,\langle v(\mu) \rangle \} $ & $v(\mu)$
                                                             as
                                                             in
                                                             proof \\
  \hline
$d$ even and $q =2$& $\{\langle v_1\rangle,\langle v_2\rangle, \langle
        v_1+v_2+v_i \rangle \mid\; 3\leq i\le d  \} $  & \\
     \hline
    \end{tabular}

    \bigskip

  For $\PGO^\varepsilon_d(q)$, let $a$ be the Witt index, let
  $w_k(\nu) = e_k-\nu f_k$, and  let $-\alpha\in \BBF$  
      be non-square.
\begin{tabular}{|l|l|l|}
\hline
$(d, \,\epsilon,\, \varepsilon)$ & $\mathcal{B}$ & Notes\\
\hline \hline
$(4,\circ, -)$  & $\{\langle x \rangle, \langle  v_1
                          \rangle, \langle e_1+ v_2       \rangle \mid
               v_1, v_2\in \langle x, y \rangle,
               \,Q(v_1)\,\mbox{and}\, Q(v_2)\,\mbox{square},\,
               |\{\langle x\rangle, \langle v_1\rangle,\langle
               v_2\rangle\}| = 3\}$ & $q \neq 3$  \\
              \hline
$(5,+,\circ)$ & $\{\langle x \rangle, \langle e_1+x \rangle,
                        \langle f_1 + x\rangle, \langle e_2+x \rangle\}$
                      & \\

                 \hline

$(5,-,\circ)$ & $\left\{ \langle w_1(\alpha) \rangle,
                         \langle  w_1(\alpha) +e_2\rangle,
                         \langle  w_1(\alpha) + f_2\rangle,
                      \langle w_2(\alpha )+ e_1\rangle,  \langle
                         w_2(1+\alpha) 
                         + f_1 +x\rangle\right\}$ &  \\
\hline
$(\ge 6,\circ,+)$   & $\left\{\langle    w_1(-1)\rangle, \langle  w_1(-1)+e_i
               \rangle, \langle w_1(-1)+f_j \rangle ,  \langle
               e_1+ w_2(-1) \rangle\mid \, 2\le i \le a, \,2\le j\le
  a-1 \right\}$& \\

 \hline
$(\ge 6,\circ,-)$   & $\{\langle x \rangle, \langle v_1 \rangle, \langle e_i+v_2 
               \rangle, \langle f_j  + x\rangle, \mid \mbox{$v_1$ and
                      $v_2$ as in $d=4$},  \, 1\le i \le a, \,
    1\le j\le a-1\} $ &$q \neq 3 $   \\

& $\small{ \{\langle x \rangle, \langle e_i + x \rangle, \langle 
       w_1(1) + y
    \rangle, \langle f_j  + x \rangle \mid 1 \leq i \le a, 1 \le j \le
    a-1\}} $&   $ q = 3 $ \\
               \hline
$(\ge 7,+,\circ)$  & $\{\langle x \rangle, \langle e_i+x \rangle,
                     \langle f_j + x \rangle\mid
  \, 1\le i \le a, \, 1\le j\le a-1\}$
                      &  \\
                      
\hline
$(\ge 7,-,\circ)$ &  $ \left\{\langle w_1(\alpha) \rangle, \langle w_1(\alpha) + e_i
  \rangle, \langle w_1(\alpha) + f_j \rangle, \langle  e_1 + w_2(\alpha) 
  \rangle,  \langle w_2(1+\alpha) + f_1 +x\rangle   \right.$ &    \\
  &\hspace{7.7cm}  $ \mid \left.2\le i \le a,\; 2\le j\le a-1\right\}$ &    \\
 \hline 
 \end{tabular}
\end{table}

\begin{proof}
The result for $\PGO^-_4(3)$ is an easy
calculation.
Let 
$H=\GO_d^\varepsilon(q)$. 
We start with $d \le 5$,
and show first that $\BB$ 
is contained in a
single $G$-orbit of the appropriate type.
For $\GO^-_4(q)$,
all $1$-spaces in $\langle x, y \rangle$ are
non-degenerate. For $q$ odd,  they are partitioned into $(q+1)/2$ spaces
$\langle v \rangle$ such that $Q(v)$ is square, and $(q+1)/2$ with
$Q(v)$ non-square. Thus for $q \neq 3$, we may find $v_1, v_2 \in
\langle x, y \rangle$ that are linearly independent, not multiples of
$x$, and such that $Q(v_i)$ is square, so that $\BB$ is a
subset of a $G$-orbit. 

For $(d,\epsilon)=(5,+),$
notice that $\langle e_1 + x \rangle^\perp = \langle e_1, x
- 2f_1,  e_2, f_2 \rangle$ is of plus type, and similarly for the rest
of $\BB$, so $\BB \subseteq \Omega$.

For
$(d,\epsilon)=(5,-)$,
notice that $\langle
w_1(\alpha) \rangle^\perp =\langle
e_1 - \alpha f_1 \rangle^\perp =  \langle e_1 + \alpha f_1, x \rangle
\oplus \langle e_2, f_2 \rangle$, and the determinant of the
restriction of the bilinear form $B$ to $\langle e_1 + \alpha f_1, x\rangle$ is $4
\alpha$, which is square if and only if $\alpha$ is square. Since
$-\alpha$ is non-square, $\alpha$ is a square if and only if $q \equiv
3 \bmod 4$, so $\langle w_1(\alpha)  \rangle \in \Omega
$ by \cite[Prop 2.5.10]{liekle}.  Similarly, the restriction of $B$ to 
$\langle w_1(\alpha) +e_2 \rangle^\perp =\langle e_1-\alpha f_1+e_2 \rangle^\perp =\langle
e_1+\alpha f_1, x, f_1-f_2, e_2 \rangle$ has determinant $-4\alpha$,
which is always non-square,
so $\langle w_1(\alpha)+e_2 \rangle \in
\Omega
 $. Notice also that
$\langle w_2(1 + \alpha)+f_1 +x \rangle^\perp =\langle e_2- (1 +
\alpha) f_2+f_1 + x \rangle^\perp =\langle e_2+ (1+ \alpha) f_2,
x-2e_1, f_1, x-2f_2\rangle $, so a short calculation shows that this $1$-space is also in $
\Omega$. The argument for
the remaining $1$-spaces is similar, so $\BB \subseteq
\Omega$.

We show next that $\BB$ is a base, so let $g\in H_{(\BB)}.$
If $d=4$ then
  the assumption that $v_2$ is not a multiple
of either $x$ or $v_1$ combines with Lemma~\ref{lem:orthogonal}(1) applied to
$(x, v_1, e_1)$ to show that $g|_{\langle x, y,  e_1 \rangle} = \pm I_3$. Since $g$ stabilises $\langle x, y \rangle^\perp = \langle
e_1, f_1 \rangle$, Lemma~\ref{lem:orthogonal}(2) shows that $g = \pm
I_4$. For $d = 5$ and $\epsilon = -$,  it is
straightforward to see that $\BB$ forms a base for $G$. For $d = 5$
and $\epsilon = +$, notice that $g$ stabilises both
$\langle x \rangle^\perp$ and
$\langle x, e_1 + x, f_1  + x \rangle$, so stabilises $\langle e_i, f_i
\rangle$ for $i = 1, 2$. It is then easy to see that $g|_{\langle x,
  e_1, e_2, f_1 \rangle} = \pm I_4$, from which
Lemma~\ref{lem:orthogonal}(2) shows that $g = \pm I_5$. 

For the rest of the proof, assume that $d \ge 6$.
In cases  $(\epsilon, \varepsilon) = (+,\circ)$ and
$\varepsilon = -$ with $q \neq 3$, the arguments that $\BB$ is contained in a single $G$-orbit of the
appropriate type, and that $\BB$ is a base for
$H=\GO^\varepsilon_d(q)$,  are identical to those for $d \leq
5$, so we will omit them. 
In the other cases, let $g \in H_{(\BB)}$. We shall show that $\BB$ is
contained in a single $G$-orbit and that $g$ is scalar.

First consider $\varepsilon = +$. Then $Q(z) = 1$ for all $\langle z \rangle \in \BB$, so
$\BB$ is contained in a single $G$-orbit.   
 Let $(w_1(-1))g=(e_1 + f_1)g = \mu (e_1 +
f_1)$, so that $\mu \in \{\pm 1\}$. For $2\le i \le a$,  
there exists $\nu_i \in \BBF_q$ such that  
$$
(w_1(-1)+e_i)g =\nu_i(w_1(-1)+e_i)=\mu(w_1(-1))+e_i g.
$$
Hence  $e_i g = (\nu_i-\mu)(w_1(-1))+\nu_i e_i$, and $Q(e_i
g)=0$ yields $\nu_i =\mu$, so $e_ig = \mu e_i$ for $i \ge 2$. 
Similarly,  $f_i g =\mu f_i$  for  $2\le
i\le a-1$, and Lemma~\ref{lem:orthogonal}(2) then yields $f_a g = \mu
f_a$. 
Since $\langle (w_2(-1)) + e_1 \rangle \in \BB$, we deduce in the same
way that $e_1g =\mu e_1$, and then $(w_1(-1))g = \mu e_1+ f_1g$
shows that $f_1g=\mu f_1$, as required.

Next consider $\varepsilon = -$ and $q =3$. 
Then $Q(y) = 2$,
so $Q(y+w_1(1))=Q(y + e_1 - f_1) 
 = 1.$ It is clear that $Q(z) = 1$ for all other $\langle z \rangle$
 in $\BB$, so $\BB$ is contained in a single $G$-orbit. 
Notice that
 $g$ stabilises $W_1:= \langle x, e_1 + x, f_1 + x\rangle= \langle x, e_1, f_1 \rangle$ and also
 stabilises $W_2 := \langle W_1,  w_1(1) + y \rangle= \langle e_1,
 f_1, x,  y\rangle$. 
 Hence $g$
 stabilises $W_1^\perp \cap W_2 = \langle x+ y \rangle$, and so
 stabilises $U:= \langle x, y \rangle$ and $U^\perp$. Then $g$ stabilises
$\langle w_1(1) + y \rangle$ and $w_1(1) g\in U^\perp,$ so
 $g$ stabilizes $\langle y \rangle.$
Lemma~\ref{lem:orthogonal}(1),
applied to both $(x,y,e_i)$ and $(x,y,f_j),$ yields  $e_ig= \mu e_i$ and $f_jg= \mu f_j$ for $1\le i\le a$ and $1\le j\le a-1.$
   The result follows from Lemma~\ref{lem:orthogonal}(1) applied to $(y,x,w_1(1))$ and
 Lemma~\ref{lem:orthogonal}(2)  applied to $(e_a,f_a)$.

Finally, consider $(\epsilon,\varepsilon) = (-,\circ)$. 
First notice that $g$ stabilises
$$V_2:= \langle
w_1(\alpha) \rangle^\perp = \langle e_1 + \alpha f_1, x, e_2, \ldots, e_a, f_2,
\ldots, f_a \rangle.$$ 
In particular
$e_ig = u_i$  for some $u_i \in
V_2$ for $2\le i \le a$. Hence there exist $\mu \in \{\pm 1\}$ and $\nu_i \in \BBF_q^\ast$ such that  
$$
(w_1(\alpha) + e_i)g =\nu_i(w_1(\alpha) + e_i)=  \mu (w_1(\alpha) )+ u_i,$$
and so 
$u_i = e_ig=\mu e_i$ for $2 \le i \le a$. Similarly,
$f_jg= \mu f_j$  for  $2\le j\le a-1.$
Then applying Lemma~\ref{lem:orthogonal}(1) to $(e_2, f_2, e_1 + w_2(\alpha))$ shows that $e_1g=  \mu e_1$
and then $f_1 g= \mu f_1$, by Lemma~\ref{lem:orthogonal}(2). 
We now deduce that
$xg\in \langle e_{a}, f_{a}, x\rangle$, and so
from $\langle w_2(1+\alpha)+ f_1 +x \rangle \in \BB$ we see that $xg
=\mu x$.  The result follows from Lemma~\ref{lem:orthogonal}(2). 
\end{proof}

We now prove that the bound for even-dimensional orthogonal groups
in Lemmas~\ref{lem:sing_one} and \ref{lem:nondeg1orth} is tight.
\begin{lemma}\label{lem:orth_sin_1_tight}
 Let $d\ge 6$ be even, and let $G =
        \PGO^\pm_d(q)$. Let $\AA =\{ \langle v_1 \rangle, \ldots, \langle
        v_{d-2} \rangle\}$ be a set of $d-2$ one-dimensional
        subspaces of the natural module $V$ for $G$.
        Then $G_{(\AA)}$ is
        nontrivial. In particular, if $\Omega$ is a $G$-orbit of
        $1$-dimensional subspaces, then $b(G, \Omega) =  d-1.$
\end{lemma}

\begin{proof}
   Let $H = \GO^\pm_d(q)$, let $W$ be
  any $(d-2)$-space containing
 $\langle \AA\rangle,$
  and let 
  $K$ denote the
  subgroup of $H$ that acts as scalars on $W$. We shall show that
there exists a nonscalar element of 
  $K$, from which the result will follow. 

If $W$ is non-degenerate, then $K$ 
contains a subgroup which acts as $\GO(W^\perp) \neq 1$ on $W^\perp$,
so the result is immediate. Thus we may assume
that $W$ is degenerate, so $U := \rad(W) = W \cap W^\perp$ is a
non-zero subspace of $W$, of dimension $1$ or $2$. 

  First assume that there exists a $u \in U$ such that $Q(u) \neq
  0$. This implies that $q$ is even, so
  $H$ has a single orbit on non-singular
  $1$-spaces, and without loss of generality we can assume that $u =
  e_1 + f_1$.
   This implies that
    $e_1, f_1 \not\in W$.
  We define $g \in \GL(V)$ by
  $$e_1 g= f_1, \quad f_1 g = e_1,  \quad zg = z \mbox{ for all } z
  \in \langle e_1, f_1 \rangle^\perp.$$
  Let $v \in V$. Then 
  $v =  \alpha e_1 + \beta f_1 + z$,
for some $\alpha, \beta \in \BBF_q$ and $z \in \langle e_1, f_1
\rangle^\perp$, and
it is easy to verify that $
Q(v  g) = Q(v)$, and so $g \in H$. 
Furthermore, if $w \in W$ then $B(w,
  e_1 + f_1)= 0$, so $w = \gamma e_1 + \gamma f_1 + z$,
  for some $\gamma \in
  \BBF_q$ (recalling that $q$ is even) and $z \in \langle e_1, f_1\rangle^\perp$. Hence $wg = w$, so
  $g \in K$, as required.

Assume instead
   that $Q(u) = 0$ for all $u \in U$, and consider first the case
   $\dim(U) = 1$.  Then we can write $W = \langle u \rangle
  \perp W'$, with $\rad(W') = 0$. If $q$ is even this contradicts the fact that
  $\dim W = d-2$ is even, so  $q$ is odd.
 There exists a $u' \in V \setminus W$ such that $B(u, u') \neq 0,$
 and we let
$W_1 = \langle W, u'\rangle$. Then $W_1$ is non-degenerate, and
$\dim(W_1) = d-1$, so $\dim(W_1^\perp) = 1$. Let $\langle z\rangle=W_1^\perp$, and define $g\in \GL(V)$ by
$$zg=-z, \quad w g=w \, \mbox{ for all } w \in W_1. $$
Let $v \in V$. Then $v = w + \alpha z$, 
for some $w \in W_1$ and $\alpha \in \BBF_q$, so
$
Q(v g)  = Q(w - \alpha z) =
Q(w) + (-\alpha)^2Q(z) = Q(v),$
            so $g \in K$, as required.
  
Finally consider the case $\dim(U) = 2.$
We fix $u_1\in U\setminus \{0\}$.
 There exists a vector
  $t_1 \in V \setminus W$ such that $(u_1, t_1)$ is a hyperbolic
    pair. 
Furthermore, 
$\langle u_1,t_1\rangle^\perp \cap U$ is $1$-dimensional, with basis $u_2$, say,
 and there exists $t_2 \in \langle u_1,t_1\rangle^\perp  \setminus W$ such that $(u_2, t_2)$ is a hyperbolic
    pair. 
  Since
            $t_1, t_2 \not\in W$, we may define an element
$g \in \GL(V)$ by
$$t_1g = t_1 + u_2, \quad t_2 g = t_2 -  u_1,  \quad w g = w \mbox{ for
  all } w \in W.$$
Let $v \in V$. Then $v = 
\alpha  t_1 + \beta t_2+w$ for some $\alpha, \beta \in \BBF_q$ and 
$w \in W$, and so
$$\begin{array}{rl}
Q(v g) & = Q(\alpha (t_1 + u_2) + \beta(t_2 - u_1)) + Q(w)   +
           B(\alpha (t_1 + u_2) + \beta (t_2 - u_1), w) \\
& = -\alpha \beta + \alpha \beta + Q(w) + B(\alpha t_1 +
    \beta t_2, w) = Q(v),
\end{array}$$
so $g \in K$, as required.
\end{proof}

\begin{table}\caption{Bases for 
    $\mathcal{N}(G, 2)$ and $\mathcal{N}^+(G, 2)$}\label{tab:k=2nonsing}
Let $V_1 = \langle e_1, f_1 \rangle, V_2 =
   \langle e_2, f_1 + f_2\rangle $, $W_i=\langle e_1 + e_i,
    f_2+f_i\rangle$, $\AA = \{V_1, V_2,  W_i \ : \ 3 \leq i \leq
   a  - 1\} $.   
\begin{tabular}{|l|l|l|}
\hline
$G$ & $\mathcal{B}$ & Notes \\
  \hline
\hline
$\PSp_6(q), \; \PGU_6(q)$ & $\{V_1, V_2 , W_3, \langle e_1 + e_2, 
                            e_1+ f_2 + f_3\rangle\}$ & $q$ odd\\
  \hline $\PGU_{2a}(q)$, $\PSp_{2a}(q)$, $\PGO^+_{2a}(q)$
    & $\AA \cup \{\langle e_2 + e_a + f_1, e_1 + f_2+f_a\rangle \}$ & $q$ even, $a \ge 4$\\
           
   & $\AA \cup \{\langle e_a + f_1, e_2 + f_1 +f_a
                \rangle \}$ &  $q$ odd,  $a \ge 4$ \\
   \hline
  $\PGO^-_{2a}(q)$  & $\AA \cup \{\langle e_1-f_1+x,
\zeta e_2-f_2+y\rangle \}$ & $a \ge 4$  \\
  \hline
  $\PGU_{2a-1}(q)$ & $\AA \cup \{\langle \lambda e_1 - f_1 + x,
\lambda e_2 - f_2 + x\rangle\}$ & $\tr(\lambda) =
                                                  1$, $a \ge 3$ \\
                                                  
  \hline
  $\PGO_{2a-1}(q)$ & $\AA \cup \{\langle  e_1 - f_1 + x ,
 e_2 - f_2 + x \rangle\}$ & $a \ge 4$ \\
 \hline
\end{tabular}
\end{table}

\begin{lemma} \label{lem:twonsing}
	Let $G \in \{\PGU_d(q), \PSp_d(q),
        \PGO_d^{\varepsilon}(q)\}$ with $d\ge 5$, and $d\ge 7$ if $G$
        is orthogonal, let $\Omega=\mathcal{N}^+(G, 2)$ when $G$
        is orthogonal and $\Omega=\mathcal{N}(G, 2)$ otherwise, and
        let $b = b(G, \Omega)$. 
        Then the set $\BB$ in
Table~\ref{tab:k=2nonsing} is a base for the action of $G$ on
$\Omega$. In particular,  if $d \ne 6$ then $b\le \lceil\frac{d}{2}\rceil$,
and if $d = 6$ then $b \le 4$.
\end{lemma}

\begin{proof}
It is straightforward to verify that the given basis of each space in
$\BB$
is a hyperbolic pair, so
$\BB \subseteq  \Omega$ in each case. 
The arguments for $d \le 6$ are similar to, but more straightforward
than, those that follow, so we shall assume
that $d \ge 7$, so that  $a = \lceil d/2 \rceil \geq 4$.
Let $H$ be $\GU_d(q)$, $\Sp_d(q)$ or
$\GO^\varepsilon_d(q)$
and let $g\in H_{(\BB)}.$

From $V_2^g=V_2$, it follows that
$(f_1+f_2)g =\beta(f_1+f_2)+\gamma e_2 = f_1 g + f_2g$
for some $\beta, \gamma \in \BBF.$ Then  $f_1g \in V_1$ and $f_2g \in
V_1^\perp$, so
equating
coefficients yields 
\[ f_1g=\beta f_1 \; \mbox{and}\;\; f_2g=\beta f_2+\gamma e_2. \]
Next, notice that $g$ stabilises $V_1^\perp \cap V_2=\langle
e_2\rangle,$ so $e_2g=\alpha e_2$, where $\beta = \alpha^{-q}$ since
$B(e_2g, f_2g) = 1$.

We shall show next that 
\begin{equation}\label{eq:gen2n}
  e_i g = \alpha e_i \mbox{ and } f_ig = \alpha^{-q}f_i  \quad \mbox{ for } 1 \leq i \leq a-1.
\end{equation} 
For $3 \le i \le a-1$, the element $g$ stabilises
$\langle V_1, V_2,
W_i \rangle \cap \langle V_1, V_2 \rangle^\perp= \langle e_i, f_i
\rangle$.
Then Lemma~\ref{lem:linone}(3), applied to 
$T = W_i$, first with $u=e_1+e_i$ and $s=f_2$, since $f_2g\in
\langle e_2,f_2\rangle,$ and then with $u=f_2+f_i$ and $s=e_1$, since $e_1g\in V_1$, 
shows that there exist $\nu_i,
\eta_i, \eta, \delta \in \BBF$ such that
$$\begin{array}{rl}
  (e_1+e_i)g & =\nu_i(e_1+e_i) =(\eta e_1+\delta f_1)+ e_ig\\
  (f_2+f_i)g & = \eta_i(f_2+f_i) =(\alpha^{-q} f_2+\gamma e_2)+ f_ig.
\end{array}$$
Equating
coefficients 
shows
that
$e_1g=\eta e_1$, $e_i g= \eta e_i$, $f_2g=\alpha^{-q} f_2,$ and $f_i g= \alpha^{-q} f_i$.
Finally, $B(e_1g,f_1g)=1$  yields
$\eta =\alpha$, and so \eqref{eq:gen2n} follows.

Finally, we apply Lemma~\ref{lem:linone}(3) to the final subspace, $T
=\langle a,b\rangle$ say, in Table~\ref{tab:k=2nonsing}. 
When $d$ is odd, setting $u=a$ and $s=e_2$ shows that $\langle a
\rangle^g =\langle a  \rangle,$ so
$\alpha = \alpha^{-q}$ and $xg = \alpha x$. 
If $d$ is even,
we deduce both that $\langle a\rangle^g=\langle a\rangle$ and $\langle b\rangle^g=\langle b\rangle,$ and hence $g=\alpha I_d.$
\end{proof}

A pair $(u, v)$ of vectors is an \emph{elliptic pair} if $Q(u) = 1$,
$Q(v) = \zeta$, for some $\zeta \in \BBF$ such that $X^2 + X + \zeta$
is irreducible, and $B(u, v) = 1$. Any elliptic pair spans
a $2$-space of minus type.

\begin{lemma} \label{lem:twonondeorth} 
Let $G =\PGO^{\varepsilon}_d(q)$, with $\varepsilon\in \{\circ, +,
-\}$
and $d\ge 7$, and let $\Omega= \mathcal{N}^-(G,2).$
    Then the set $\BB$ in
Table~\ref{tab:k=2minustype} is a base for the action of $G$ on
$\Omega$. Consequently, $b(G, \Omega)\le \lceil\frac{d}{2}\rceil$. 
\end{lemma}

\begin{table}\caption{Bases for 
    $\mathcal{N}^-(G, 2)$}\label{tab:k=2minustype}
 Let $V_1 = \langle e_1+f_1, e_2+ f_1 + \zeta f_2\rangle$.
 If $\zeta \neq 1$ then let $V_2 =  \langle e_2+f_1 + f_2,  e_1+\zeta
 f_1 \rangle$, otherwise let $V_2 = \langle e_1 + f_1 + f_2,  e_2 + f_2\rangle$.
 Let $W_i=\langle e_1+e_i + f_1, e_2+ \zeta e_i+f_i \rangle$ and
 $\AA = \{V_1, V_2,  W_i \ : \ 3 \leq i \leq
    a  - 1\} $.    
    \begin{tabular}{|l|l|}
      \hline
      $G$ & $\mathcal{B}$ \\
      \hline  \hline

  $\PGO_{2a-1}(q)$ & $\AA \cup \{V_3 =\langle e_2+f_1+x, e_1+e_3+\zeta
                     f_3 \rangle\} $ \\
      \hline
$\PGO^+_{2a}(q)$
    & $\AA \cup \{V_{4}= \langle e_1+e_2+f_2+f_a,  e_3+f_1+\zeta f_3\rangle  \}$ \\                                  \hline
  $\PGO^-_{2a}(q)$& $\AA \cup \{V_5 =\langle e_1-e_3+x, f_1+f_3+y\rangle \}$ \\
    \hline
\end{tabular}
\end{table}

\begin{proof}
Fix  $\zeta$ such that $X^2 + X + \zeta$ is
irreducible, with $\zeta = Q(y)$ if $\varepsilon = -$.
Let $a = \lceil d/2 \rceil\ge 4$. 
One may check
that the given ordered
basis vectors 
for each $2$-space in $\BB$
form an elliptic
pair, and so 
$\BB \subseteq \Omega$. For example, $Q(e_1 + f_1) =
1$, $Q(e_2 + f_1 + \zeta f_2) = B(e_2, \zeta f_2) = \zeta$, and $B(e_1
+ f_1, e_2 + f_1 + \zeta f_2) = B(e_1, f_1) = 1$. 
Let $H=\GO^\varepsilon_d(q)$, and let $g \in
H_{(\BB)}$. 
To show that $\BB$ is a base for $G$, it suffices
to show that $g$ is scalar.

We shall show first that there exists $\alpha = \pm 1$ such that
\begin{equation}\label{eq:gen2nondegorth2}
(e_1 + f_1)g= \alpha(e_1 + f_1), \quad e_i g= \alpha e_i \mbox{ for }
2 \le i \le a-1, \quad  f_i g= \alpha f_i \quad \mbox{for} \;3\le i\le a-1.
\end{equation} 	

Let $U = \langle V_1, V_2 \rangle = \langle e_1, e_2, f_1, f_2\rangle$.
Then
$g$ stabilises $U_i:= \langle U,
W_i \rangle$ for $3 \le i \le a-1$, and so stabilises $U_i \cap U^\perp = \langle e_i, f_i \rangle$.
Since $g$ stabilises $W_i$ and $V_1$, 
 there exist $\mu_i,\nu_i, \alpha, \nu \in \BBF$ such that
\begin{align*}
(e_1+e_i+f_1)g&=\mu_i(e_1+e_i+f_1)+ \nu_i (e_2+\zeta e_i+ f_i) \in W_i\\
& = (e_1+f_1)g+e_ig =\alpha (e_1+ f_1)+ \nu (e_2+ f_1+\zeta f_2)+ e_ig.
\end{align*}
Since $e_i g \in \langle e_i, f_i \rangle$, looking at $f_2$, we see that
$\nu=0$. Hence $\nu_i=0$,  and so $(e_1+f_1)g=\alpha(e_1+f_1)$ and $e_i g = \alpha e_i$
for $3 \le i \le a-1$. 
We now apply Lemma~\ref{lem:orthogonal}(2) to
$(e_i,f_i)$ to see that
$f_ig=\alpha^{-1} f_i$ for $3 \le i \le a-1$. 

To prove \eqref{eq:gen2nondegorth2}, 
it remains to prove
that $e_2g= \alpha e_2=\pm e_2$.
Considering $W_3$ shows that
$$(e_2+\zeta e_3+f_3)g= e_2 g + \zeta e_3 g + f_3 g 
= e_2g + \zeta  \alpha e_3 + \alpha^{-1} f_3 = \lambda (e_2+ \zeta e_3+
f_3)  + \mu (e_1 + e_3 + f_1),$$ for some $\lambda, \mu \in \BBF$.
Then $e_2g \in U$, so considering $f_3$ gives $\lambda  = \alpha^{-1}$. Then considering
$e_3$ yields $\mu = \zeta(\alpha - \alpha^{-1})$, and so $e_2 g =
\alpha^{-1} e_2 + \zeta(\alpha - \alpha^{-1})(e_1 + f_1)$. Finally,
$Q(e_2 g) = 0$ shows that $\alpha = \alpha^{-1} = \pm 1$ and so
\eqref{eq:gen2nondegorth2} is verified. 

\medskip

Let $\AA \subseteq \BB$ be as in
Table~\ref{tab:k=2minustype}, and let $X = \langle \AA \rangle$.
Then
$g$ stabilises $X^\perp$.  
If we can show that either of $e_1g = \alpha e_1$ or $f_1 g = \alpha
f_1$, then it follows from \eqref{eq:gen2nondegorth2} that the same is
true for the other. In particular, this will imply that 
$\langle
e_1, f_1 \rangle^g= \langle e_1, f_1\rangle$. It will then follow from
$U^g = U$ that
$\langle e_2, f_2 \rangle^g = \langle e_2, f_2 \rangle$. Hence, it
will follow from Lemma~\ref{lem:orthogonal}(2) applied to $(e_2,
f_2)$ that $f_2 g=\alpha^{-1}f_2 = \alpha f_2$. 
Hence, to show that
$g$ is scalar it suffices to show that $vg = \alpha v$ for either $v
= e_1$ or $v  = f_1$,  and for whichever of $v \in \{e_a, f_a, x, y\}$
is defined. We shall use \eqref{eq:gen2nondegorth2}
implicitly. 

\medskip

If $d$ is odd, then Lemma~\ref{lem:linone}(3) with $T=V_3$ and
initially with $u = e_1+e_3+\zeta f_3$ and  $s = x$,
gives
$ug= \alpha u$, and so $e_1g=\alpha e_1$ and hence $f_1 g = \alpha f_1$.
Now, setting $u = e_2 + f_1 + x$ and $s = e_1$ shows that  $u g=
\alpha u$.  Hence $x g= \alpha x$, and so $g$ is scalar. 

\smallskip

If $\varepsilon = +$ then
Lemma~\ref{lem:linone}(3) applied to  $T
= V_4$, $u=e_3+f_1+\zeta f_3$ and $s=f_a\in X^\perp$
shows that $ug=\alpha u.$ Hence
$f_1 g= \alpha f_1$ and so $e_1 g=
\alpha e_1$. 
Next, letting $s = f_1$ shows that
$f_a g= \alpha f_a$. Finally, Lemma~\ref{lem:orthogonal}(2) applied to
$f_a, e_a \in X^\perp$ proves that $g$ is scalar.

\smallskip

Finally, if $\varepsilon = -$ then $X^\perp=\langle x,y \rangle.$
Lemma~\ref{lem:linone}(2), applied to $T = V_5$ with
$s = f_3$  yields 
$(e_1-e_3+x)g =\alpha(e_1-e_3 + x)$, and applied again with $s = e_3$
shows that 
$(f_1+f_3 + y)g =\alpha(f_1+f_3+y)$. 
Hence 
$g=\alpha I_d$.
\end{proof}	

\subsection{Symplectic groups on the cosets of orthogonal groups}

We consider $\Sp_{d}(q)$, acting
on the cosets of $\GO^{\pm}_d(q)$, with $q$ even.

\begin{table}\caption{Bases for $H = \GO_{2m+1}(q) \cong
    \Sp_{2m}(q)$ on $\NN^{\pm}(H, 2m)$, with $q$ even}\label{tab:symplecticoonorthogonal}
\noindent
Let $X^2 + X + \lambda^2$ be
irreducible.  Let $A_i=\langle e_i, f_i \rangle$
and $B_i =
\langle e_i + x, f_i + \lambda x\rangle$.
  \begin{tabular}{|l|l|}
   \hline
   $\NN$ & $\BB$ \\
  \hline\hline 
$\NN^+(H, 2m)$
    &
 $\left\{T = \bigoplus_{i = 1}^m  A_i, \  \ U_i = A_1  \oplus    \dots  \oplus A_{i-1}\oplus
      \langle e_i, f_i+x \rangle \oplus A_{i+1}\oplus \dots \oplus
      A_m \right.$\\
    & $
      V_{j} = A_1 \oplus \cdots
      A_{j-1} \oplus \langle e_j+x, f_j \rangle  \oplus  A_{j+1}
      \oplus \cdots \oplus A_m \mid \ 1\le i \le
      m, 1\le j \le
      m-1\}$\\
  \hline
$\NN^-(H, 2m)$
    & 
      $\left\{
      T= B_1 \oplus (\bigoplus_{i = 2}^m
      A_i), \ \ U_{i} = B_1 \oplus A_2\oplus   \dots  \oplus
      A_{i-1}\oplus  \langle e_i, f_i+x \rangle \oplus A_{i+1}\oplus
      \dots \oplus A_m  \right. $\\
   & $V_j =B_1 \oplus A_2\oplus
     \dots  \oplus A_{i-1}\oplus  \langle e_j+x, f_j \rangle \oplus
     A_{i+1}\oplus \dots \oplus  A_m, $ \\
   & $W_1 = A_1 \oplus  B_2 \oplus A_{3}\oplus \dots
     \oplus A_m,  \ \   W_2 =\langle e_1, f_1+x\rangle\oplus  B_2
    \oplus A_{3}\oplus \dots \oplus A_m$ \\      
& $\hspace{8cm}  \mid \ 2 \le i \le m, \ 2 \le j \le m-1\}$ \\                               
  \hline
\end{tabular}
\end{table}

\begin{prop} \label{prop:exactbasesize}
  Let $G=\Sp_{2m}(q)$ with $2m \ge 6$ and $q$ even, and let $M=\GO^{\pm}_{2m}(q).$
  Then $b(G, M\setminus G)= 2m$. 
\end{prop}
\begin{proof}
We shall use the isomorphism $\Sp_{2m}(q) \cong \GO_{2m+1}(q)$ to
consider the equivalent actions of $H = \GO_{2m+1}(q)$ on
$\NN^{\pm}(H, 2m)$, where the natural module for $H$ is
$V = \BBF_q^{2m+1}$ with quadratic form
$Q$ as in Definition~\ref{def:standardform}. That is, we shall
consider the actions of $H$ on non-degenerate
$2m$-dimensional subspaces of $+$ and $-$ type, since the point
stabiliser of $H$ in these actions is $\GO^{\pm}_{2m}(q)$.

We shall first show that the set $\BB$ in
Table~\ref{tab:symplecticoonorthogonal}
is a base for $H$, and then show that $\BB$ is of minimal size. First
notice that $(e_i, f_i)$, $(e_i, e_i + f_i + x)$, and $(e_i + f_i + x, f_i)$ 
are hyperbolic pairs, therefore $A_i$,  $\langle e_i,
f_i + x \rangle$ and $\langle e_i + x, f_i \rangle$ are $2$-spaces of
$+$ type. The basis of $B_i$ is
an elliptic pair, so in each case $\BB \subseteq \Omega$.

 Let $g \in
H_{(\BB)}$. Then we shall show that $g = 1$. From
$Q(x) = 1$ and $\langle x \rangle = \rad(V) = V \cap V^\perp$, we deduce that $xg =
x$. 
We first consider $\NN^-(H, 2m)$. 
For $2 \le i \le m$, the element $g$ stabilises 
$$T \cap U_i= B_1 \oplus A_2\oplus
\dots  \oplus A_{i-1}\oplus \langle e_i \rangle \oplus  A_{i+1}\oplus
\dots  \oplus A_{m}, $$ and so stabilises
$\rad(T\cap U_i)=\langle e_i\rangle$. Hence there exists
$\alpha_i \in \BBF_q$ such that 
$e_ig=\alpha_i e_i$, for $2\le i\le m$. 
Similarly,  
 $\rad(T \cap V_i)^g = \rad(T \cap V_i)$ 
 so
$f_ig= \alpha_i^{-1} f_i,$ for $2\le i\le m-1$.
Since $m \ge 3$, 
the space
$S:= \langle A_2,  \dots, A_{m-1}\rangle$ is non-degenerate and 
stabilised by $g$, so $g$ also stabilises 
 $$S^\perp \cap W_1=
\langle A_1, 
A_m, x\rangle \cap W_1 = \langle A_1, A_m\rangle.$$
Hence $g$ stabilises $\langle A_1,  A_m\rangle \cap W_2 = \langle
e_1, A_m\rangle$, and so fixes the radical of this space, which is $\langle e_1\rangle$.
 Moreover, $g$ stabilises $\langle e_1, A_m \rangle \cap T = A_m = \langle e_m, f_m
 \rangle$. Hence, since $e_mg =
 \alpha_m e_m$,  
 Lemma~\ref{lem:orthogonal}(2) shows
that $f_mg=\alpha_m^{-1}
 f_m$. In addition, $g$ stabilises  $A_m^\perp \cap \langle A_1,
 A_m \rangle = A_1 
 $, and 
 Lemma~\ref{lem:orthogonal}(2) now yields 
 $f_1g=
 \alpha_1^{-1}f_1$. Next,  for $2 \le i \le m$ we deduce from $(f_i + x)g
 = \alpha_i^{-1} f_i + x \in U_i$ that $\alpha_i = 1$, and the same
 follows for $\alpha_1$ from $(f_1 +x)g \in W_2$.
 
The arguments for $\NN^+(H, 2m)$ are very similar but
easier. Consideration of $T \cap U_i$ shows that for all $i$ there exists
an $\alpha_i \in \BBF_q$ such that $e_i g= \alpha_i e_i$.
Then an identical argument applied to $T \cap V_j$ shows that $f_i g =
\alpha_i^{-1} f_i$ for $i \leq m-1$.  Therefore, $g$
stabilises $\langle
A_1,\dots,A_{m-1}\rangle^\perp\cap T = \langle e_m, f_m \rangle$, and so
$f_mg = \alpha_m^{-1}f_m$, also. 
Finally, notice that $(f_i+ x)g =
\alpha^{-1}_i f_i + x \in U_i$ for 
all $i$, and hence $\alpha_i =
1$, as required.

It remains only to show that these bases are of minimal size.
Let $\AA:=\{T, S_1,\dots ,S_{2m-2}\}$ be a set of $2m-1$
non-degenerate $2m$-spaces of $V$ of sign $\varepsilon$ (either $+$
or $-$). We shall show that $H_{(\AA)}\ne 1$.
The stabiliser in $H$ of $T$ is $H_T=\GO^\varepsilon_{2m}(q)$, which
acts naturally on $T$ as $\GO^\varepsilon_{2m}(q)$. It suffices to show
that the stabiliser in $H_T$ of all of the spaces $T \cap S_i$ for $1 \leq i \leq
2m-2$ is nontrivial.

Since $\dim(T\cap S_i) = 2m-1$, the restriction of $B$ to
$T \cap S_i$ is degenerate,  and so $T\cap S_i$
has a one-dimensional radical $\langle v_i\rangle$. Hence the
2-point stabiliser $H_{T, S_i}$ stabilises $\langle v_i \rangle \leq
T$. Furthermore, since $T$ is non-degenerate, it follows that
$\dim(v_i^\perp \cap T) = 2m-1$, and so $T \cap S_i = T
\cap v_i^\perp$. Hence $H_{T, S_i}$ is equal to $H_{T,\langle v_i\rangle}$ and so $H_{(\BB)} = \cap_{i = 1}^{2m-2} H_{T, S_i}$
contains $(H_T)_{\langle
v_1 \rangle, \ldots, \langle v_{2m-2} \rangle}$. This group is nontrivial by
Lemma~\ref{lem:orth_sin_1_tight}. 
 \end{proof}

\section{Proof of Theorem~\ref{thm:main} for almost simple
  groups}\label{sec:almost_simple}

In this section, we shall prove the following theorem, which in
particular implies Theorem~\ref{thm:main} for almost simple groups.

\begin{thm}\label{thm:almost_simple}
Let $G \le \Sym(\Omega)$ be a primitive almost simple group of degree $n$
that is not large base. If $b(G) > \lceil \log n \rceil + 1$, then
$G = \mathrm{M}_{24}$, $n = 24$ and $b(G) = 7$.  Furthermore, if $b(G) \geq
\log n + 1$ then $(G, n, b(G)) \in \{(\mathrm{M}_{12}, 12, 5),
(\mathrm{M}_{23}, 23, 6), (\mathrm{M}_{24}, 24, 7)\}$ or $G =
\PSp_{2m}(2)$ with $m \ge 3$,  $n = 2^{2m}
  - 2^{m - 1}$ and $b(G) =  2m = \lceil \log n \rceil + 1$.
\end{thm}

We shall first consider the standard actions of $\alt(\ell)$ and $\sym(\ell)$ on
partitions, then the actions of the classical groups on totally
singular and non-degenerate $k$-spaces, and (for
the orthogonal groups in even characteristic) non-singular
$1$-spaces. Then we shall look at the action of groups with
socle $\PSp_d(2^f)$ on the cosets of the normaliser of
$\GO^{\pm}_d(2^f)$, before considering the remaining subspace
actions. Finally, we will deal with the
non-standard actions, and hence prove
Theorem~\ref{thm:almost_simple}.

\subsection{Action on partitions}\label{sec:partition}
We first consider the non-large-base standard actions of
$\alt(\ell)$ and $\sym(\ell)$. 

\begin{thm}\label{thm:an_sn}
	Let $s \ge 2$ and $t \ge 2$, with $\ell:= st \ge 5$, and let
$G$ be $\sym(\ell)$. Let $\Omega$ be the set
of partitions of $\{1,2,\dots, \ell\}$ into $s$ subsets
of size $t$, and let $n=|\Omega|$. 
Then $b:=b(G, \Omega)< \log n + 1.$
\end{thm}
\begin{proof} The degree $n$  is $\ell!/{(t!)}^ss!$. 
If $t = 2$, then $s \ge 3$ and $n \ge
  \frac{6!}{2^3 \cdot 3!} = 15$. However, $b = 3$  by \cite[Remark 1.6(ii)]{BGS11}, so $b < \log n$. 
If $s \ge t \ge 3$, then $n \ge \frac{9!}{(3!)^3 3!} =
280$, whilst $b \le 6$ by~\cite[Theorem 4(i)]{standard1}, so again $b <
\log n$.  

For the remaining cases, 
by~\cite[Theorem 4(ii)]{standard1}
\begin{equation}\label{eq:an_bound}
b\le\lceil\log_s t\rceil+3\le \log_s t+4 = \log_s (\ell/s)+4=\log_s\ell+3.
\end{equation}

Next consider $s=2$, so that $t \ge 3$. 
We check in \MAGMA \cite{Magma} that if $t = 3$, $4$, $5$ 
then $b$ is at most $4$, $5$, and $5$, respectively,
whilst $n = 10$, $35$, $126$, and so $b < \log n + 1$ in each
case. Assume therefore that $\ell \geq
12$. Then
\[
	n =\frac{\ell!}{{2((\ell/2)!)}^2}=\frac{\ell(\ell-1)\dots(\ell-\ell/2+1)}{2(\ell/2)!}\\
	 =\frac{\ell(\ell-1)\dots(\ell/2+2)(\ell/2+1)}{(\ell/2)(\ell/2-1)\dots
          2\cdot 2}>2^{\ell/2}.
\]
	In particular, since $\ell\ge 12$, we deduce from
        \eqref{eq:an_bound} that
\[
b\leq \log \ell + 3 < \frac{\ell}{2}+1 \le	\log n+1.
\]

Next, let $s = 3$. We may assume that $t > s$, so 
$\ell \geq 12$. Then, reasoning as for $s = 2$,
we deduce that $n \geq 2^{\ell/3} \cdot 3^{\ell/3} > 2^{2
  \ell/3}$. Hence $\log n > 2\ell/3$, so \eqref{eq:an_bound} yields
\[
b \leq \log \ell + 3 <\frac{2\ell}{3} + 1 < \log n + 1.
\]

We are therefore left with $4 \le s < t$, so that $\ell
\ge 20$.
For  all $\ell$, the groups $\alt(\ell)$ and $\sym(\ell)$ have no
core-free subgroups of index less than
$\ell$, 
so
$\ell < n$. From \eqref{eq:an_bound} we deduce that 
\[ b\le \log_s\ell + 3= \frac{\log \ell}{\log s}+3\le \frac{\log \ell}{2}+3  \leq \log \ell+1 < \log n+1.\]
\end{proof}

\subsection{Subspace actions }\label{sec:subspace}

We now prove Theorem~\ref{thm:almost_simple}
for the subspace actions of almost simple
groups.  First we record two lemmas concerning base size and automorphism groups.

\begin{lemma}\label{lem:regularcycles}
Let $G$ be a finite almost simple primitive permutation group on
$\Omega$ with socle $G_0$ a non-abelian simple classical group,
and let $G_0 \unlhd G_1 \unlhd G \le \Omega$. If $G/G_1$ has a subnormal series of length $s$ with all quotients cyclic, then $b(G)\le b(G_1)+s$.
\end{lemma}
\begin{proof}
  If $(G_0,\Omega)$ is not isomorphic to
  $(\alt(\ell), \{1, \dots, \ell\})$ with $\ell\in \{5, 6,8\}$, then
by \cite[Theorem 1.2]{GuestSpiga17} each element of $G$ has a
regular cycle. It follows that stabilising one point for each cyclic
quotient suffices to extend a base for $G_1$ to one for $G$.

Let $(G_0,\Omega)$ be $(\alt(\ell), \{1, \dots, \ell\})$ for some
$\ell\in \{5, 6,8\}$. Then 
$G \le \sym(\ell)$, since the other automorphisms of
$\alt(6)$ do not act on $6$ points, and so 
the result follows from
$b(\sym(\ell), \Omega) = b(\alt(\ell), \Omega) + 1$.
\end{proof}

Recall the meaning of $\SS(G, k)$ and $\NN^{\epsilon}(G, k)$ from
Notation~\ref{not:SandN}. The following bounds are established in
\cite[Proof of Theorem 3.3]{HaLiMa}. 

\begin{lemma}\label{lem:HLMbounds}
   Let $G_0$ be a simple classical group.
  \begin{enumerate}
   \item[(i)] If $(G_0, k)
     \neq (\POmega^+_{2m}(q), m)$ then $b(G_0, \mathcal{S}(G, k)) \leq
     d/k + 10$. 
   \item[(ii)] $b(G_0, \mathcal{N}(G, k)) \leq d/k + 11$.
     \item[(iii)] If $G_0 = \PSL_d(q)$ then $b(G_0, \mathcal{S}(G, k))
       \leq d/k + 5$. 
   \end{enumerate}
 \end{lemma}

\begin{lemma} \label{lem:orth_small}
Let $G$ be almost simple with socle
$G_0$ one of $\POmega_4^-(q)$, $\POmega_5(q)$ or 
$\POmega^\pm_6(q)$. Assume that $G_0 \neq \POmega_4^-(q)$ for $q \le 3$.
Let $\Omega = \NN^\epsilon(G, 1)$, with $\epsilon = +, -$ or blank, and let $n = |\Omega|$. 
Then $b:= b(G) < \log n + 1$.
\end{lemma}

\begin{proof}
First let $d = 4$, so that $q \ge 4$. Then  $\Aut(G_0)/\PGO_4^-(q)$
has a normal series of length at most $2$ with all quotients cyclic,
so Lemmas~\ref{lem:nondeg1orth}
and \ref{lem:regularcycles} imply
that $b \leq 5$. 
From \cite[Table 4.1.2]{burngiu} we see that 			
$n = \frac{q (q^2+1)}{(q-1, 2)} > 2^6$, so the result follows. 

Next let $d = 5$, so that $q$ is odd. 
Then $\Aut(G_0)/\PGO_5(q)$ is cyclic, so     
Lemmas~\ref{lem:nondeg1orth} and \ref{lem:regularcycles} yield
$b \leq 6$, whilst from \cite[Table 4.1.2]{burngiu} we see that
$n= \frac{q^2(q^{4}-1)}{2(q^{2}\mp1)} \ge 36 > 2^5,$
so the result follows.

Finally, let $d = 6$.
From Lemmas~\ref{lem:nondeg1orth} and \ref{lem:regularcycles} we
deduce that  $b \leq 5$ if $q = 2$, and $b \le 7$ otherwise. 		
Moreover, by \cite[Table 4.1.2]{burngiu}, 
\[
   n \ge \frac{q^{2}(q^3 +  1)}{(q-1, 2)} > \left\{
    \begin{array}{ll}
      2^5 & \mbox{ if $q = 2$}\\
      2^6 & \mbox{ if $q \ge 3$}
     \end{array} \right.
\] as required. 
\end{proof}

\begin{prop}\label{prop:ortheven}
  Let $d \ge 8$ be even, and let
  $G$ be almost simple  with socle $G_0 =
\POmega^{\varepsilon}_d(q)$. Let $\Omega$ be $\SS(G, k)$ or $\NN^\epsilon(G,
k)$, where $\epsilon = +$, $-$ or blank,  and let $n
= |\Omega|$.  If $G$ acts primitively on $\Omega$ then $b:= b(G) < \log n + 1$.
\end{prop}

\begin{proof}
We shall use throughout the proof the fact that if $G_0 \neq
\POmega_8^+(q)$, or if $\Omega \neq \SS(G, 2)$ (see, for
  example, \cite[Table 8.50]{Cbook}), 
then $G/G_0$ has a normal series with all
quotients cyclic of length at
most three, and less if $q = 2$ or $q$ is prime, so
that $b(G) \le b(G_0) + 3$, by
Lemma~\ref{lem:regularcycles}. Furthermore, under the same
conditions, $b(G) \leq b(\PGO_d^\varepsilon(q)) + 2$, by the
same lemma. If $G_0 = \POmega_8^+(p^f)$ and $p$ is odd then
$\Out(G_0) \cong \Sym(4) \times C_f$, whilst if $p = 2$ then
$\Out(G_0) \cong \Sym(3) \times C_f$. 

  First consider 
  $\Omega = \SS(G, k)$. By  \cite[Tables 3.5E and F]{liekle} we may
assume that $1 \le k \le d/2$, and $k \le d/2-1$ if $\varepsilon =
-$. 
If $k \le d/2-1$ then  by \cite[Table 4.1.2, Cases VI and VII]{burngiu}
\begin{align}
 \label{eq:orth_sing_kspace} n	& =\frac{ (q^{\frac{d}{2}}\mp1) (q^{\frac{d}{2}-k} \pm 1)
   \prod_{i=\frac{d}{2}-k+1}^{\frac{d}{2}-1}(q^{2i}-1)}{\prod_{i=1}^{k}(q^i-1)}\\
   & \ge
   \frac{(q^{\frac{d}{2}}\mp1 )(q^{\frac{d}{2} -1} \pm 1)}{q-1}\prod_{i=2}^{k}(q^{i}+1)
                   > 
                   q^{d-k-1}\prod_{i=2}^{k}q^{i}
               \label{eq:big_k}      \ge q^{d-2+\frac{k(k-1)}{2}}.
\end{align}
If 
  $k = 1$ then Lemma~\ref{lem:sing_one} shows
  that $b \le d+1$, with tighter bounds when $q \le 3$, whilst
  \eqref{eq:big_k} gives $n > q^{d-2}$, so $\log
n  + 1> (d-2) \log q + 1 \ge b$.  
Similarly, if $k = 2$ 
then we deduce from Lemma~\ref{lem:twosing} 
that $b \leq 
d/2 + 3$ if $q$ is even or prime, and $b \le d/2 + 5$ in general.  
From (\ref{eq:big_k}), we see that $\log n+1 > (d-1) \log q + 1$, so
the result follows. 

Next we consider the case $3 \le k \le d/2-1$, so that
$b(G_0, \Omega) \leq d/k + 10$ by Lemma~\ref{lem:HLMbounds}(i). 
First assume that $q \le 3$ and $k= 3$. We calculate in \MAGMA that if $(d, q) = (8,
2)$ then $b \le 4$. For $(d, q) = (8, 3)$ we use
the exact value of $n$ and the fact that $b(G) \le b(G_0) + 2$ 
to see that $\log n + 1 >  15 \ge
b$. For $d \ge 10$ we
see from \eqref{eq:orth_sing_kspace} that
\[
n  
 \ge  \frac{
      (q^\frac{d}{2}\mp1)(q^{\frac{d}{2}-3}\pm1)}{q-1}(q^4+1)(q^{2}+1)(q^3+1)>
      q^{d-4}q^{4+3+2}\ge q^{d+5}.
\]
Hence if $q = 2$ then $\log n+1 \ge d+6\ge \frac{d}{3}+11 \ge b$,
and if $q = 3$ then $\log n+1 \ge \frac{3}{2}(d+5)+1\ge \frac{d}{3}+13
\ge b$.
In the remaining cases $k \ge 4$ or $q \ge 4$, so 
the result follows by a routine calculation from (\ref{eq:big_k}).

Finally consider $k = d/2$, so that $\varepsilon = +$. From \cite[Table 4.1.2]{burngiu},
$ n = 	\prod_{i=1}^{\frac{d}{2}}(q^i+1)\ge
        \prod_{i=1}^{\frac{d}{2}}q^i=q^{\frac{d(d+2)}{8}} \ge q^{10}.$
It is shown in \cite{HaLiMa}
        that $b(G_0) \leq 9$, so $b \leq 10$ when $q = 2$, and $b  \leq 12$
        otherwise, and the result follows. 

\medskip

We now consider  $\Omega = \NN^{\epsilon}(G, k)$, with $\epsilon \in
\{+, -\}$ or blank.  The stabiliser of an element of $\Omega$ also
stabilises a non-degenerate $d-k$ space, of the
opposite sign if  $\varepsilon = -$ and $k$ is even, and of the same sign
otherwise. Thus by considering the stabiliser of spaces of 
type $+$,  $-$  and $\circ$,  we may assume that $k \leq d/2$. 

First assume that $k$ is even, so that $2
\le k \le  d/2$, and if $k = d/2$ then $\varepsilon = -$, by our
assumption that $G$ acts primitively. Then we
deduce from \cite[Table 4.1.2, Cases X, XI, XIII]{burngiu} (by replacing $d$ by $d-k$ if $\epsilon = +$ and $\varepsilon =
-$) that
\begin{align}\label{eq:orth_kspacenondeg}
  n 
  =&\frac{q^{\frac{k(d-k)}{2}}(q^{\frac{d}{2}} - \varepsilon)\prod_{i=\frac{d-k}{2}}^{\frac{d}{2}-1}(q^{2i}-1)}{2(q^{\frac{k}{2}}
     - \epsilon)(q^{\frac{d-k}{2}}- \varepsilon
     \epsilon)\prod_{i=1}^{\frac{k }{2}-1}(q^{2i}-1)}
     = \frac{q^{\frac{k(d-k)}{2}}(q^{\frac{d}{2}} -
     \varepsilon)(q^{d-2} - 1)}{2(q^{\frac{k}{2}}
     - \epsilon)(q^{\frac{d-k}{2}}- \varepsilon
     \epsilon)} \prod_{i = 1}^{\frac{k}{2}-1}\frac{q^{d-k-2+2i} -
     1}{q^{2i} - 1}
\end{align}
If $k = 2$ then it follows  that  
$n  >  q^{2d-6} \ge q^{d+2}$, whilst from
Lemmas~\ref{lem:twonsing} and \ref{lem:twonondeorth} we see that
$b \le d/2 + 2 < \log n + 1$. 
For $k \ge 4$, notice that
\begin{align*}
 n
\nonumber & > \frac{q^{\frac{k(d-k)}{2}}(q^{\frac{d}{2}}- \varepsilon)(q^{d-2} -
                   1)}{4(q^{\frac{d}{2}} - 1)} \prod_{i =
                   1}^{\frac{k}{2} -1} q^{d-k-2}
\geq \frac{1}{4}q^{\frac{k(d-k)}{2} + {d-3} + 
                                                     (d-k-2)(\frac{k}{2}-1)}
 = \frac{1}{4}q^{kd - k^2 - 1}.    \end{align*}
 The quadratic $kd - k^2 - 1$ attains its minimum for $4 \le k \le
 d/2$ at $k = 4$, so $\log n + 1 > (4d - 17) \log q - 1\ge 4d-18.$
Then by Lemma~\ref{lem:HLMbounds}(ii), 
$b \leq \frac{d}{4} + 14.$ If $d\ge 10$, then $\log n + 1 \ge 4d-18\ge
\frac{d}{4} + 14$, so it
only remains to consider $(d,k,\varepsilon)= (8,4,-)$. In this case, 
$(4d - 17) \log q - 1\ge 15\log q+1$ and the result follows easily for $q\ge 3.$
If $q=2$ then $\Out(G_0)$ is cyclic, hence by Lemmas~\ref{lem:regularcycles} and \ref{lem:HLMbounds}(ii),  $b \leq 14$ and the result follows.

Now let $k$ be odd,  so without loss of generality $1\le k<d/2$.  
By \cite[Table 4.1.2, Cases IX, XII, XIV]{burngiu}
\begin{align*}
n=&\frac{q^{\frac{(kd-k^2-1)}{2}}(q^{\frac{d}{2}}- \varepsilon
    )\prod_{i=\frac{d-k+1}{2}}^{\frac{d}{2}-1}(q^{2i}-1)}{(2,
    q-1)\prod_{i=1}^{\frac{k-1 }{2}}(q^{2i}-1)}. 
\end{align*}
If $k=1$ then $n> \frac{q^{d-2}}{(2, q-1)}$, whilst 
Lemma~\ref{lem:nondeg1orth} shows that
$b\le d+1$, with tighter bounds when $q \le 3$, so the result follows easily.
If $ k \ge 3$ then $q$ is odd,  with $b \le d/3 + 14$ by Lemma~\ref{lem:HLMbounds}(ii). 
Now
\begin{align*}
  n& = \frac{1}{2}q^{\frac{(kd-k^2-1)}{2}}(q^{\frac{d}{2}}- \varepsilon)
     \prod_{i = 1}^{\frac{k-1}{2}}\frac{q^{d-k-1+2i} - 1}{q^{2i}-1} >
  \frac{1}{2}q^{\frac{(kd-k^2-1)}{2} + \frac{d}{2}-1} \prod_{i
     = 1}^{\frac{k-1}{2}} q^{d-k-1} \\
& \quad \quad \quad  \quad \quad \quad \quad = \frac{1}{2} q^{\frac{kd-k^2- d - 3}{2} 
     + \frac{k-1}{2}(d-k-1) } = \frac{1}{2}q^{kd - k^2-1}\ge \frac{1}{2}q^{3d - 9-1},
 \end{align*}    
 where the last inequality follows as in the case $k$ even, so the
 proof is complete.
\end{proof}

\begin{prop}\label{prop:orthodd}
Let $d\ge 7$ and let $G$ be almost simple with socle $G_0 =
\POmega^{\circ}_d(q)$. Let $\Omega$ be $\SS(G, k)$ or $\NN^\pm(G,
k)$,  and let $n
= |\Omega|$. If $G$ acts primitively on $\Omega$ then $b:= b(G) < \log n + 1$.
\end{prop}

\begin{proof}
We shall use throughout the proof the fact that $\Out(G_0)$ has a
normal series with at most two cyclic quotients, and
$\Aut(G_0)/\PGO_d(q)$ is cyclic, so $b \le b(G_0, \Omega)
+ 2$ and $b \le b(\PGO_d(q), \Omega) + 1$, by
Lemma~\ref{lem:regularcycles}. 
  
First let $\Omega = \SS(G, k)$. Then $1 \le k \le (d-1)/2$ and by
\cite[Table 4.1.2, Case VII]{burngiu} 
\begin{align}\label{eq:deg_orth_odd_sing}
n&	=\frac{\prod_{i=\frac{d-2k+1}{2}}^{\frac{d-1}{2}}(q^{2i}-1)}{\prod_{i=1}^{k}(q^i-1)}.
\end{align}	
If $ k \le (d-3)/2,$ then
\begin{align*} 
n & \ge \frac{(q^{d-1}-1)}{(q-1)}\cdot\frac{(q^{2k}-1)\dots(q^{4}-1)}{(q^k-1)\dots (q^2-1)}
	\ge q^{d-2}q^{\sum_{i=2}^{k}i}=q^{d-3+\frac{k(k+1)}{2}}.
\end{align*}
If $k = 1$ then $n > q^{d-2} \ge 3^{d-2}$, and
from Lemma~\ref{lem:sing_one} we
deduce that 
$b\le d+1$, as required. 
If $k = 2$ then $n\ge q^d$,  whilst
Lemmas~\ref{lem:twosing} give $b\le \lceil d/2 \rceil + 1$. 
If $3 \le k \le (d-3)/2$ then $d\ge 9$ and $n \geq q^{d+3}$. 
Hence $
	\log n+1\ge (d+3)\log q+1\ge 3d/2+ 11/2\ge d/3+12 \ge
        b$, by Lemma
        \ref{lem:HLMbounds}(i). 
        
	Finally, assume that $k=(d-1)/2$, so that
        \eqref{eq:deg_orth_odd_sing} simplifies to
        $
	n
	=\prod_{i=1}^{\frac{d-1}{2}}(q^i+1)\ge
          q^{\frac{d-1}{4}(\frac{d+1}{2})} = q^{(d^2-1)/8}$. 
For $(d, q) \in \{(7, 3), (7,5)\}$ the result follows from a \MAGMA
calculation. Otherwise, by 
Lemma \ref{lem:HLMbounds}(i), 
	$b\le \frac{d}{(d-1)/2}+10 + 2 <  3 +12=15,$
so we are done.

Now let $\Omega = \NN^{\pm}(G, k)$, so that without loss of generality
$k$ is even. Then by \cite[Table 4.1.2, Cases XV and XVI]{burngiu}
\begin{align*}
n  &=\frac{q^{\frac{k(d-k)}{2}}\prod_{i=\frac{d-k+1}{2}}^{\frac{d-1}{2}}(q^{2i}-1)}{2(q^{\frac{k}{2}}\mp1)
    \prod_{i=1}^{\frac{k}{2}-1}(q^{2i}-1)} 
   = \frac{q^{\frac{k(d-k)}{2} } (q^{d-1} - 1)}{2 (q^{\frac{k}{2}} \mp
    1)}  \prod_{i = 1}^{\frac{k}{2}-1}
    \frac{q^{d-k-1+2i} - 1}{q^{2i}-1}\\
  & \geq \frac{1}{4} q^{\frac{k(d-k)}{2} + (d-1-\frac{k}{2}) + (d-k-1)(\frac{k}{2}  - 1)} = \frac{1}{4} q^{kd-k^2}
\end{align*}
If $k=d-1$ then $n \ge \frac{1}{4}q^{d-1} \ge
\frac{1}{4}3^{d-1}$, whilst $b = b(G, \NN^{\pm}(G, 1))  \leq d < \log
n + 1$, by
Lemma~\ref{lem:nondeg1orth}.
Similarly, if $k= 2$ then $n \ge \frac{1}{4}q^{2d-4}>q^{d}$, whilst 
$b \leq \lceil d/2 \rceil+1 < \log n + 1$ by
Lemmas~\ref{lem:twonsing} and \ref{lem:twonondeorth}. 
For $4 \le k \le d-3$, the quadratic $- k^2 +kd$ attains its minimum
at $k = d-3$, so $\log n + 1 \ge (3d-9) \log q - 1.$
Now,
Lemma~\ref{lem:HLMbounds}(ii) yields $b \le d/4 + 13$,
which is less than $(3d+13)/2\le d\log q+(2d-9)\log q-1= \log n+1$,
so the proof is complete.  
\end{proof}

\begin{prop}\label{prop:sp}
  Let $d \ge 4$, and let $G$ be almost simple with socle
  $G_0 = \PSp_d(q)$, with $(d, q) \neq (4, 2)$. 
Let $\Omega$ be $\SS(G, k)$ or $\NN(G, k)$, and let $n = |\Omega|$.
If $G$ acts primitively on $\Omega$, then $b:= b(G) < \log n + 1$.
\end{prop}

\begin{proof}
  We shall use throughout the proof the fact that
  $\Out(G_0)$ has a normal series with at most two cyclic
  quotients, so 
  $b(G, \Omega) \le b(G_0, \Omega) + 2$, with $b(G, \Omega) \le  b(G_0,
  \Omega) + 1$ if $q > 2$ is even or prime, by
  Lemma~\ref{lem:regularcycles}. 
  
  First let $\Omega = \SS(G, k)$. Then  $1 \le k \le d/2$, and by 
  \cite[Table 4.1.2]{burngiu}
\begin{equation}\label{eq:symp_deg}
n =\frac{\prod_{i=\frac{d}{2}-k+1}^{\frac{d}{2}}(q^{2i}-1)}{\prod_{i=1}^{k}(q^{i}-1)}
          = \prod_{i = 1}^k \frac{q^{d-2k+2i} - 1}{q^i-1}.
\end{equation}
If $k =1$ then $n=(q^d-1)/(q-1)>q^{d-1}$. By Lemma~\ref{lem:sing_one},
$b \le d+2$, with $b \le d$ if $q = 2$ and $b \le d+1$ if $q = 3$. 
The result now follows from a straightforward calculation, since $d
\ge 4$.

If $k = 2$ and $d \ge 6$ then  $b \le d$, by
Lemma~\ref{lem:twosing}, whilst $\log n + 1 > (2d-5) + 1 \ge  b$. 
If  $k = 2$ and $d = 4$ then Lemma~\ref{lem:twosing}
  implies that $b(G_0, \Omega) \le 4$ and a routine calculation shows that
  $b < \log n + 1$.

  If $k \ge 3$ then Lemma~\ref{lem:HLMbounds}(i) yields
	$	b\leq \frac{d}{k}+12$, 
with  $b \leq \frac{d}{k} + 11$ when $q \le 8$. 
First suppose that $d - 2k\ge 2$, so that $d \ge 8$.   If $(d, q) =
(8, 2)$ then we verify the result in \MAGMAn. Otherwise, we notice that
$n > |\SS(\PGO^\pm_d(q), k)|$, and our upper
bounds on $b$ are less than the corresponding
bounds for the orthogonal groups, so the result follows by the same
calculations as in the
proof of Proposition~\ref{prop:ortheven}.
We may therefore assume that $k=\frac{d}{2}$, so that  $b \le 14$
in general, and $b \le 13$ if $q \le 8$. In
this case $n = 	\prod_{i=1}^{\frac{d}{2}}(q^i+1)\ge
q^{\frac{d(d+2)}{8}}$, so if
$d \ge 10$ then the
result is immediate.
For $d = 6$ and $q \le 4$, a \MAGMA calculation establishes the
result, whilst if $q \ge 5$ then $\log n + 1 > 14 \ge b$.  
For $d = 8$, if  $q=2$ then $\log n + 1 > 12 \ge b$, whilst for
$q \geq 3$ we deduce that 
	$\log n +1 > 16 > b$. 

        \medskip

Next let $\Omega = \NN(G, k)$. Then $k$ is even  and without loss of
generality $k \leq d/2 - 1$. By 
\cite[Table 4.1.2]{burngiu} 
\[
	n =\frac{q^\frac{k(d-k)}{2}\prod_{i=\frac{d-k+2}{2}}^{\frac{d}{2}}(q^{2i}-1)}{\prod_{i=1}^{\frac{k}{2}}(q^{2i}-1)}.
\]
        If $k = 2$ then $d \ge 6$ so $\log n + 1 \ge ((d-2) +
        (d-2))\log q + 1 > d
        \ge b$, by Lemma~\ref{lem:twonsing}. 
        If  $k\ge4$ then from $d\ge 2k+2$ we deduce that
        $q^{\frac{k(d-k)}{2}} \ge q^{\frac{k(k+2)}{2}} \ge q^{12}$ and 
        $(q^{d-k+2} -
       1) > q^2(q^2-1)(q^k-1)$, so 
$\prod_{i=\frac{d-k+2}{2}}^{\frac{d}{2}}(q^{2i}-1)\ge
(q^{d}-1)q^2\prod_{i=1}^{\frac{k}{2}}(q^{2i}-1)$. Putting these
together shows that 
$	n\ge
q^{12}(q^{d}-1)q^2 >  q^{d+13}$, so the result follows
from 
 Lemma~\ref{lem:HLMbounds}(ii).
\end{proof}

\begin{prop}\label{prop:uni}
Let $d \ge 3$,  let $G$ be almost simple with socle
  $G_0 = \PSU_d(q)$, 
let $\Omega$ be $\SS(G, k)$ or  $\NN(G, k)$, and let $n = |\Omega|$.
Then $b:= b(G) < \log n + 1$.
\end{prop}

\begin{proof}
  We shall use throughout the proof the facts that
  $\Aut(G_0)/\PGU_d(q)$ is cyclic, whilst $\Out(G_0)$ has a normal
  series with two cyclic quotients, so
  $b(G, \Omega) \leq b(\PGU_d(q), \Omega)+1$ and $b(G, \Omega) \leq
b(G_0, \Omega) + 2$,  by
Lemma~\ref{lem:regularcycles}. 

First let $\Omega = \SS(G, k)$. 
Then by \cite[Table 4.1.2]{burngiu}
\[ n
         =\frac{\prod_{i=d-2k+1}^{d}(q^{i}-(-1)^i)}{\prod_{i=1}^{k}(q^{2i}-1)}
       = \prod_{i = 1}^k \frac{(q^{d - 2k + 2i - 1} - (-1)^{d-1})(q^{d-2k +
          2i} - (-1)^{d})}{q^{2i} - 1}, 
\]
so $n > q^d$ if $k = 1$, $n > q^{2d-4}$ if $k = 2$,   and
\begin{equation}\label{nunitsing2}
n \geq \prod_{i = 1}^k (q^{d - 2k +
            2i - 1} + 1) \geq q^{(d-1)+(d-3)+(d-5)}=q^{3d-9} \mbox{ if
            $k \ge 3$.}
	\end{equation}
If $k = 1$ then $b \leq d+1 \leq \log n + 1$ by
Lemma~\ref{lem:sing_one}, so
let $k = 2$. If $(d, q) = (4, 2)$ then a
\MAGMA calculation shows the result, and otherwise if $d = 4$ then  $b
\le 6 \le 4 \log q + 1 < \log n + 1$, by
Lemma~\ref{lem:twosing}, as required.
If $d \ge 5$ then $b \le d < \log n +1$  by Lemma~\ref{lem:twosing}. 

Finally, let $k \ge 3$, so that $d \ge 6$. 
For $(d, q) \in \{(6, 2), (6, 3), (7, 2), (7, 3)\}$ we verify the
result computationally. Otherwise,  $b \leq \frac{d}{k}+12 \le d/3  +
12$ by
Lemma~\ref{lem:HLMbounds}(i). 
If $d \ge 8$  then \eqref{nunitsing2} gives 
$\log n+1\ge 3d-8\ge \frac{d}{3}+12 \geq b,$
as required.
Similarly, if $q\ge 4$ then
$\log n+1\ge 2(3d-9) +1 \ge d/3+12 \geq
b$, which covers all the remaining cases. 

\medskip

Now let $\Omega = \NN(G, k)$. 
Then by \cite[Table 4.1.2]{burngiu}
\begin{align*}
n=\frac{q^{k(d-k)}\prod_{i=d-k+1}^{d}(q^{i}-(-1)^i)}{\prod_{i=1}^{k}(q^{i}-(-1)^i)}.
\end{align*}
If $k \le 2$ then $n > q^d$, and the result follows easily from
Lemma~\ref{lem:nondeguni1} and
Lemma~\ref{lem:twonsing}.
For  $k \ge 3$, we get
$\prod_{i=d-k+1}^{d}(q^{i}-(-1)^i)\ge (q^d-(-1)^d) \prod_{i=1}^{k}(q^{i}-(-1)^i)
$, because $d\ge 2k+1\ge 7$.  Hence
$
n\geq q^{k^2+k}(q^{d}-(-1)^d)\ge q^{d+11},
$ and the result follows from Lemma~\ref{lem:HLMbounds}(ii).
\end{proof}

\begin{prop}\label{prop:psl} Let $d\ge 2$ and when $d=2,$  let $q\ge 7$.  
  Let $G$ be almost simple with socle
  $G_0 = \PSL_d(q)$,
let $\Omega=\SS(G, k)$, and let $n = |\Omega|$.
Then $b:= b(G) < \log n + 1$. \end{prop}

\begin{proof}
The group $\Out(G_0)$ has a normal series with all quotients cyclic of
length at most three, and $G/G_0$ has such a series with length at
most two if $k \neq d/2$, or if $d = 2$, or if $q$
is prime; and is cyclic if more than one of these conditions hold. 
Hence by Lemma~\ref{lem:regularcycles}, $b \le b(G_0, \Omega) + \ell$, where $\ell = 3$ in
  general, but with smaller values of $\ell$ for the special cases above.

First let $k = 1$, so that $n = (q^d-1)/(q-1) > q^{d-1}$, whilst
$b\le d+2$ by Lemma~\ref{lem:linone}, with smaller bounds if $q\le 3.$
The result follows from a lengthy but straightforward calculation,
using $n = q+1 \ge 8$ when $d = 2$. 

If $k = 2$ then $n 
> q^{2d-4}$.  If $d = 4$ and $q \le 3$ then a \MAGMA calculation
             shows that $b \le  5 < \log n$, and if  $d = 4$ and $q > 3$ then $b
             \le 5 + 2 < \log n$,  by
              Lemma~\ref{lem:twospacelin}.  If $d
              > 4$ then  $b \le \lceil d/2
              \rceil + 3 < \log n + 1$ by Lemma~\ref{lem:twospacelin}.

Assume finally that $d/2 \ge k \ge 3$, so that $d \ge 6$, and 
\[
n
=
\frac{q^d-1}{q-1} \cdot \frac{q^{d-1}-1}{q^2-1}  \cdot \frac{q^{d-2} -
  1}{q^3-1} \prod_{i=1}^{k-3}
\frac{(q^{d-k+i}-1)}{(q^{i+3}-1)} > q^{(d-1) + 3 + 1} = q^{d+3}.
\]
Then from Lemma~\ref{lem:HLMbounds}(iii), we deduce that $b \leq d/3 +
8\le d+4 \le
\log n + 1$.
\end{proof}

We now meet the unique infinite family of examples that
attains the upper bound in Theorem~\ref{thm:main}.

\begin{prop}\label{prop:sp_on_orth}
  Let $q = 2^f$, let $d=2m\ge 4$, and let  $G$ be almost simple with socle $G_0=\Sp_{d}(q)$. Assume that $(d,q)\neq (4,2)$.  Let
  $M=N_G(\GO^{\epsilon}_{d}(q))$, let $\Omega=M\setminus G,$ let $n=|\Omega|$, and let $b=b(G)$.

  If $\epsilon = -$ and $q = 2$ then
  $\log n + 1 < b=\lceil \log n\rceil+1.$
  Otherwise, $b< \log 
n+1.$ 
\end{prop} 

\begin{proof}
We calculate that $n=|\Sp_d(q):\GO_d^{\epsilon}(q)|=q^{m}(q^{m} + \epsilon)/2$.
If $q=2$ then $b= 2m$ by
Proposition~\ref{prop:exactbasesize}.  If
$\epsilon = +$ then
$n >2^{2m-1},$ hence $\log n+1>  b.$ 
If $\epsilon = -$ then 
$  \lceil \log n\rceil+1
=2m=b$.

It is proved in \cite{HaLiMa} that
$b(G_0, \Omega)
\leq 2m+1$, so $b \leq 2m+2$
by Lemma~\ref{lem:regularcycles} since $\Out(G_0)$ is cyclic.
Therefore if $q \ge 4$ then 
$$\log n+1 >  \log (q^{2m-1}/2)+1 =  (2m-1)\log q\ge 4m - 2 \ge
b,$$ and the proof is complete. 
\end{proof}

Our final result in this subsection deals with all of the remaining subspace actions.

\begin{prop}\label{prop:nov}
  Let $G$ be an almost simple classical group, with a primitive subspace action
  on a set $\Omega$ of size $n$, with point stabiliser $H$. Assume that $\Omega$ is not a $G$-orbit of
  totally singular, non-degenerate, or 
   non-singular subspaces, and that if $G_0 = \soc(G) =
  \Sp_{2m}(2^f)$ then $(G_0 \cap H)\neq \GO^{\pm}_{2m}(2^f)$.
  Then $b:=b(G) < \log n + 1$.
\end{prop}

\begin{proof}
  Definition~\ref{subspaceaction} implies
  that $G$ is not simple, and $H$ is a novelty maximal
subgroup of $G$. 
Consulting \cite{liekle} and \cite{Cbook}, we see that one of the following holds:
\begin{itemize}
\item[(i)] $G_0=\PSL_d(q)$, $d\ge 3$ and $G \not\leq \PGamL_d(q)$;
\item[(ii)] $G_0=\PSp_4(q)$, $q$ even and $G \not\leq \PCGamSp_4(q)$;
\item[(iii)] $G_0=\POmega_8^+(q)$ and $G \not\leq
  \mathrm{PC\Gamma O}^+_8(q)$ (in the notation of \cite[Table
  1.2]{Cbook}). 
\end{itemize}
In particular, from \cite{Cbook}, in each case there exists a group $G_1$ such that
$G_0 \unlhd G_1 \unlhd G$, the quotient $G/G_1$ has a normal series of length at most two
with all quotients cyclic, and $H \cap
G_1$ is a subgroup of
the stabilizer $H_1$  in $G_1$ of a totally singular
 $k$-space,  of index greater than four.  Let
$\Omega_1$ denote the right coset space of $H_1$ in $G_1$ and let
 $b_1= b(G_1,
\Omega_1).$ 
Then there exist $x_1,\dots, x_{b_1}\in G_1$ such that
$H_1^{x_{1}}\cap \dots\cap H_1^{x_{b_1}}$ is trivial, so 
$H^{x_{1}}\cap \dots\cap H^{x_{b_1}}\cap G_1$ is also trivial. By
Lemma~\ref{lem:regularcycles},  $b \leq b_1 + 2$ . 
 
Finally, notice that $n \geq 4|\Omega_1|$ by the Orbit-Stabiliser Theorem, so
if $b_1 < \log |\Omega_1| + 1 = \log 2|\Omega_1| \leq \log
n - 1$, then $b  < \log  n + 1$.  
The result is now immediate from
Propositions~\ref{prop:psl}, \ref{prop:sp} and \ref{prop:ortheven}.
\end{proof}
	
\subsection{Proof of Theorem~\ref{thm:almost_simple}}

\begin{proof}[Proof of Theorem~\ref{thm:almost_simple}]
  Let $G_0 = \soc(G)$.
 The only non-large-base almost simple primitive groups of
  degree $n \leq 8$ are the actions of $\Alt(5)$ and $\Sym(5)$ on $6$
  points, of $\PSL_3(2)$ on $7$ points,
  and of $\PSL_2(7)$ and $\PGL_2(7)$ on $8$ points, all of
  which have base size $3$, which is less than $\log n
  +1$. Hence the result holds for $n \le 8$, and therefore for $b(G) \leq
  4$.

Since the groups $\PSL_2(q)$ are isomorphic to many other simple
groups, we shall consider them next. If $G_0 \cong \PSL_2(5)$ then all actions either
have degree at most $6$ or are large base, so
let $G_0$ be $\PSL_2(q)$ for $q \ge 7$, and let 
$H=G_\omega,$ for some $\omega \in \Omega.$
We work through the choices for $H$, as described in \cite[Table
8.1]{Cbook}. The result for $H \in \mathcal{C}_1$ follows from
Proposition~\ref{prop:psl}. Burness shows in \cite[Table
3]{Burness07} that $b(G) \le 3$ for the majority of the remaining
choices of $H$. More precisely, he shows that  $b(G) \leq 3$ if $H \in \mathcal{C}_2 \cup
\mathcal{C}_3$, or if $H \in \mathcal{C}_5$ and $q=q_0^r$ with $r\neq 2$,
or if $H \in \mathcal{C}_6$ and $q  > 7$; or if $H \in
\mathcal{C}_9$ and $q \neq 9$. We therefore need consider only the
exceptions with $q \ge 7$. 
If $H \in
\mathcal{C}_5$ and 
$q=q_0^2$, then $q_0 \ge 3$ and the action of $G_0$ on
$\Omega$ is equivalent to that of $\POmega^-_4(q_0)$ on non-degenerate
$1$-spaces. If $q_0 = 3$ then $G_0 \cong \Alt(6)$,
 and the action is equivalent to the (large base) action on
$2$-sets. Hence we can assume that $q_0 \ge 4$, and the 
result follows from Lemma~\ref{lem:orth_small}. 
If either $H\in \mathcal{C}_6$ and $q =7$, or
$H\in \mathcal{C}_9$ and $q=9$, then $n\le 7$,  so the result follows.
 Thus for the remainder of the proof we
shall assume that $G_0 \not\cong \PSL_2(q)$.

Next,  assume that the action of $G$ is not  standard.
  Burness, Guralnick and Saxl show in \cite{BGS11} that if
$G_0 \cong \An$ then $b(G) \leq 3$. 
For classical groups $G$, Burness shows in
\cite[Theorem 1.1]{Burness07} that either $n=1408$ and $b(G) = 5$ 
or $b(G) \le 4$.
For the exceptional groups $G$, it is shown by Burness, Liebeck and
Shalev in~\cite{LIETYPE}, that $b(G) \leq 6$; since the
smallest degree of a faithful primitive representation of an
exceptional group is $65$ (see, for example, \cite[Table
B.2]{permutationgroups}), the result follows. Finally, Burness,
O'Brien and Wilson show in \cite{BurnessOBrienWilson10} that if $G$ is
sporadic,
then either $b(G) \leq 5$,
 or $G$
is $\mathrm{M}_{23}$, $\mathrm{M}_{24},$  $\mathrm{Co}_3$,
$\mathrm{Co}_2$, or $\mathrm{Fi}_{22}.2$, with a specified action.
If $\log n+1 \le 5,$ then $n\le 16$, and 
the only sporadic group with a faithful primitive action on at
most $16$ points, other than $\mathrm{M}_{12}$ as given
in the statement, is $\mathrm{M}_{11}$ on $11$ or $12$ points, with
base size $4.$
The actions of $\mathrm{M}_{23}$ and $\mathrm{M}_{24}$ are given
in the theorem statement,
whilst the remaining actions
 have base size $6$ and
very large degree. 

It remains to consider the standard actions that are not large base.
If $G_0 = \Alt(\ell)$, then $\Omega$ is an orbit of partitions of $\{1,
\ldots, \ell\}$, so $b(G) \leq \log n + 1$ by Theorem~\ref{thm:an_sn}.
Hence we may assume that $G$ is a classical group in a subspace
action.

If $G_0 = \PSL_d(q)$ then the result follows from
Propositions~\ref{prop:psl} and \ref{prop:nov}. If $G_0 = \PSU_d(q)$
then we may assume that $d \ge 3$, and the result follows from
Propositins~\ref{prop:uni} and \ref{prop:nov}.

If $G_0 = \PSp_d(q)$ then we may assume that $d \ge 4$, and $(d, q)
\neq (4, 2)$, since $\PSp_4(2)' \cong \PSL_2(9)$. 
If the action
is
on $k$-spaces then the result follows from
Proposition~\ref{prop:sp};
if $q$ is even and the point stabiliser is
$\GO^\pm_d(q)$, then it follows from
Proposition~\ref{prop:sp_on_orth}; and otherwise it follows from
Proposition~\ref{prop:nov}.

%

If $G_0 = \POmega^\varepsilon_d(q) $ then our assumption that $G_0
\not\cong \PSL_2(q)$ implies that $d \ge 5$, so assume first that $d
\in \{5, 6\}$, and let $H_0$ be whichever of $\PSp_4(q)$, $\PSL_4(q)$
or $\PSU_4(q)$ is isomorphic to $G_0$. 
If the action is on totally singular subspaces, then the
action of $G_0$ is equivalent to that of $H_0$
on totally singular subspaces. If the action is on non-degenerate
$2$-spaces, then the action of $G_0$ is equivalent to that
of $H_0$
on the maximal subgroups in Class $\mathcal{C}_2$ or $\mathcal{C}_3$,
and $b(G) \leq 3$ by \cite[Table 3]{Burness07}.  If 
 the action is on an orbit of non-degenerate $1$-spaces, 
 then the result follows from
 Lemma~\ref{lem:orth_small}, and otherwise it follows from Proposition~\ref{prop:nov}.
Hence we may assume that $d \geq 7$, and the result follows
from Propositions~\ref{prop:ortheven}, \ref{prop:orthodd} and \ref{prop:nov}.
\end{proof}

\section{Proof of Theorem~\ref{thm:main}} \label{sec:alltheremaining}

In this section, we prove  Theorem~\ref{thm:main}.

\begin{prop}\label{prop:SD}
Let $G\le \sym(\Omega)$ be a primitive group of diagonal type and degree $n$. Then
$b:=b(G)\leq \max \{4, \log \log n \}.$ In particular, $b <\log
n$.
\end{prop}

\begin{proof}
  Let $\soc(G) = T^k$, where $T$ is a non-abelian simple group and $k\geq
2$.
Then $n=|T|^{k-1}$ and we may assume that $G =  T^k.(\out(T)\times
\sym(k))$. For the final claim, notice that $n \geq 60$, so $\log n >
4$, and so it suffices to prove the first claim. 

If  $k = 2$ then $b\le 4$, as proved by Fawcett in \cite{diagonal}.
It is also proved in \cite{diagonal} that if $k
\ge 3$ then
\begin{equation}\label{eqdia}
b \leq \biggl\lceil \frac{\log k}{\log|T|}\biggr\rceil+2.
\end{equation}
If $3 \leq k \le |T|$ then $b \le 3$ and the result follows, so
assume that $k > 60$. Then $n \geq 60^{60}$, so $\log \log n > 8$, and hence
$$b \leq \frac{\log k}{\log 60} + 3 \leq \frac{\log \log n}{5} + 3
\leq \log \log n.$$
\end{proof}
We now consider product action type groups.

\begin{prop}\label{prop:product_action}
  Let $G
  \le \sym(\Omega)$ be a primitive
group of product action type and degree $n$. 
  If $G$ is not large base
  then  $b:=b(G)< \log
  n+1$.  
\end{prop}

\begin{proof}  Without loss of generality, we may assume that $G=H\wr \sym(k),$  where
 $H\le  \sym(\Gamma)$ is primitive, and either $H$ is almost simple and not large base 
or $H$ is of diagonal type.  Let $|\Gamma|=m,$ so $n=m^k.$
Let $\{\gamma_1, \dots, \gamma_c\} \subseteq \Gamma$ be a base of
minimal size for the action of $H$ on $\Gamma$, and let
$\alpha_i':=(\gamma_i,\dots,\gamma_i)\in \Gamma^k=\Omega$ for $1\le
i\le c$.  It is shown in the proof of \cite[Proposition 3.2]{BuSe} that there
exists a set of $\lceil \log k \rceil$ $2$-partitions of $\{1,
\ldots, k\}$ such that the intersection in $\sym(k)$ of the
stabilizers of these partitions is trivial.  Let $a=\bigl\lceil \log k
\bigr\rceil$ and $r=\bigl\lfloor \log m \bigr\rfloor$. Then, as in the proof of \cite[Lemma 3.8]{BuSe},
there exists a subset $\{\alpha_1,\dots,\alpha_{\lceil
  a/r\rceil}\}$ of $\Omega$ with the property that an element
$g\in G$ which factorizes as $g=(1,\dots,1)\sigma,$ where $1\in H$ and $\sigma\in \sym(k),$ fixes each $\alpha_i$ if and only if $\sigma=1$. 
Hence, as noted in \cite[Equation (13)]{BuSe}, the set 
$$\BB:=\{\alpha_1,\dots,\alpha_{\lceil a/r\rceil}\}\cup 
\{\alpha_1',\dots,\alpha_c'\}$$ is a base for $G.$ 
 In particular, we deduce that
 \begin{equation}\label{eq:prod}b\le  \biggl\lceil \frac{\bigl\lceil\log
  k\bigr\rceil}{\bigl\lfloor \log m\bigr\rfloor
}\biggr\rceil+b(H,\Gamma). 
\end{equation}

From Theorem~\ref{thm:almost_simple} and Proposition~\ref{prop:SD},
we see that  either $b(H,\Gamma)\le \lceil\log
m\rceil+1 \le \log m + 2$, or $(H, m, b(H, \Gamma)) =
(\mathrm{M}_{24}, 24, 7)$.
In this lattter case
$$b\le  \biggl\lceil \frac{\lceil \log k \rceil}{\bigl\lfloor \log m\bigr\rfloor
}\biggr\rceil+b(H,\Gamma) \leq  \left( \frac{1 + \log k}{4} +1 \right)
+7 < k\log (24)+1\le \log n+1.$$
For the general case, assume first that $k \le 4$, so
that in particular $\lceil \log k \rceil \leq \lfloor \log m \rfloor$. Then by~\eqref{eq:prod}
$$b \leq 1 + b(H, \Gamma) \le  \log m + 3 < 2 \log m +
1 \le k \log m + 1 = \log n + 1.$$
If instead $k \ge 5$, then
\begin{align*}
b & \leq \biggl\lceil \frac{\lceil \log k \rceil}{\lfloor
               \log m \rfloor} \biggr\rceil+  \log m+2 
             \leq \frac{ 1 + \log k}{\lfloor \log m \rfloor} + \log m + 3
            \leq \left( \frac{1 + \log k}{2} + 2 \right) + \log m + 1\\
             & < (k-1) + \log m + 1
             < k \log m  + 1 = \log n + 1
\end{align*}
as required.
\end{proof}

Finally, we state and prove a slightly more detailed version of
Theorem~\ref{thm:main}.

\begin{thm}\label{thm:final}
Let $G$ be a primitive subgroup of $\sym(\Omega)$ with $|\Omega|=n$. Assume that
$G$ is not large base. Then $b:= b(G) \geq \log n + 1$ if and only
if $G$ is
one of the following.
\begin{enumerate}
\item[(i)] A subgroup of $\AGL_d(2)$, with $b = d + 1 = \log n + 1$.
\item[(ii)] The group $\Sp_d(2)$, acting on the cosets of $\GO^-_d(2)$ with $d\ge 4$, for
  which $\log n + 1 < b  = \lceil \log n \rceil +1$. 
\item[(iii)] A Mathieu group $\mathrm{M}_n$ in its natural permutation
  representation with $n\in \{12,23,24\}$. If $n = 12$ or $23$ then $b = \lceil \log n
  \rceil + 1$, while if $n = 24$ then $b = 7 > \lceil \log n \rceil
  + 1$.
\end{enumerate}
\end{thm}

\begin{proof}
  We work through the cases of the O'Nan-Scott Theorem.

  If $G$ is of affine type, then without loss of generality $G=
  \AGL_d(p)$ with $n = p^d$, and the point stabiliser of $G$ is
  $\GL_d(p)$, acting naturally on the set $\Omega = \BBF_p^d$. 
Let $\BB$ be a base of minimal size for $\GL_d(p)$ on $\Omega$. Then
$\BB$ is a basis for $\BBF_q^d$, so $b = |\BB| + 1 = d+1$ as
required. 

If $G$ is of twisted wreath product type, then by
\cite[Section 3.6]{18} the group $G$ is a subgroup of a
primitive product action  group $H \wr P \le \sym(\Omega)$,
with $H$ of diagonal type. Hence the result follows from 
 Proposition~\ref{prop:product_action}.

If $G$ is almost simple, or of diagonal type, or of product action
type,
then the result follows from
Theorem~\ref{thm:almost_simple}, Proposition~\ref{prop:SD} or
Proposition~\ref{prop:product_action}, respectively.
\end{proof}

We conclude with a question.

\begin{question}
  Which primitive groups $G \leq \Sym(n)$ satisfy $b(G) = \log n+ 1$?
\end{question}

Notice that such a $G$ must be a subgroup of $\AGL_d(2)$ for some $d$,
and that if $d$ is even then groups such as $2^{d}:\Sp(d,2)$ have this property. 

\bibliographystyle{amsplain}

\end{document}